\documentclass[a4paper,12pt]{article}
\usepackage{latexsym,amsmath,amsthm,amssymb}
\usepackage{a4wide}
\usepackage{hyperref}
\usepackage{marginnote}
\usepackage{color}
\usepackage{cite}
\hypersetup{
pdftitle={Leray--Adams}   
pdfauthor={Van Hoang Nguyen},
colorlinks = true,
linkcolor = magenta,
citecolor = blue,
}

\theoremstyle{plain}
\newtheorem{theorem}{Theorem}[section]

\newtheorem*{Theorema}{Theorem A}
\newtheorem*{Theoremb}{Theorem B}
\newtheorem*{Theoremc}{Theorem C}
\newtheorem*{Theoremd}{Theorem D}

\newtheorem{proposition}[theorem]{Proposition}

\newtheorem{corollary}[theorem]{Corollary}

\theoremstyle{definition}

\theoremstyle{remark}

\renewcommand{\thefootnote}{\arabic{footnote}}

\def\R{\mathbb R}

\def\al{\alpha}
\def\om{\omega}
\def\Om{\Omega}
\def\de{\delta}
\def\De{\Delta} 

\def\vphi{\varphi}
\def\ep{\epsilon}
\def\na{\nabla}
\def\la{\langle} 
\def\ra{\rangle} 
\def\lt{\left}
\def\rt{\right}

\numberwithin{equation}{section}

\title{The Leray--Adams inequality}
\author{Van Hoang Nguyen\footnote{
Institute of Mathematics, Vietnam Academy of Science and Technology, 18 Hoang Quoc Viet, 10307 Cau Giay, Hanoi, Vietnam.}
}

\begin{document}
\maketitle


\renewcommand{\thefootnote}{}

\footnote{Email: \href{mailto: Van Hoang Nguyen <vanhoang0610@yahoo.com>}{vanhoang0610@yahoo.com} and \href{mailto: Van Hoang Nguyen <nvhoang@math.ac.vn>}{nvhoang@math.ac.vn}}

\footnote{2010 \emph{Mathematics Subject Classification\text}: 26D10, 46E35, 35A23, 26D15.}

\footnote{\emph{Key words and phrases\text}: Hardy--Rellich inequality, Adams inequality, Leray--Adams inequality}

\renewcommand{\thefootnote}{\arabic{footnote}}
\setcounter{footnote}{0}

\begin{abstract}
In this paper, we establish the following Leray--Adams type inequality on a bounded domain $\Om$  in $\R^{4}$ containing the origin,
\[
\sup_{u\in C_0^\infty(\Om), \tilde I_4[u,\Om,R] \leq 1}  \int_\Om \exp\lt(c\lt( \frac{|u|}{E_2^{\beta}\lt(\frac{|x|}R\rt)}\rt)^2\rt) dx \leq C |\Om|
\] 
for some constants $c >0$ and $C >0$, where $\beta\geq 1$, $R \geq \sup_{x\in \Om} |x|$,
$
\tilde I_4[u,\Om,R]:= \int_\Om |\De u|^2 dx -  \int_\Om \frac{|u|^2}{|x|^{4} E_1^2\lt(\frac{|x|}R\rt)} dx,
$
and $E_1(t) = 1-\ln t$, $E_2(t) = \ln (eE_1(t))$ for $t \in (0,1]$. This extends the Leray--Trudinger inequality recently established by Psaradakis and Spector \cite{PS2015} and Mallick and Tintarev \cite{MT2018} to the case of Laplacian operator. In the higher dimensions or higher order derivatives, we prove the Leray--Adams type inequality for radial function on the ball $B_r$ (with center at origin and radius $r >0$) in $\R^n$.
\end{abstract}

\section{Introduction}
The main aim of this paper is to establish some Leray--Adams type inequalities which are closely related  to different types of the Moser--Trudinger inequality, the Adams inequality and the Hardy--Rellich inequalities. Let us quickly recall some relevant results about these inequalities.

Let $\Om$ be a bounded domain in $\R^n$. The famous Sobolev embedding theorem says that for any $p < n$
\[
W_0^{1,p}(\Om) \hookrightarrow L^q(\Om),\quad\text{\rm for any $q$ satisfying}\quad 1\leq q \leq p^* =\frac{np}{n-p}.
\]
In the limiting case $p =n$, we still have $W_0^{1,n}(\Om)\hookrightarrow L^q(\Om)$ for any $q < \infty$. However, the embedding $W^{1,n}_0(\Om) \hookrightarrow L^\infty(\Om)$ fails. A simple counter-example is as follows: suppose $x_0\in \Om$ and the ball $B_r(x_0)$ with center at $x_0$ and radius $r>0$ is included in $\Om$, then the function $w(x) = \ln (-\ln (r/|x-x_0|)) \chi_{B_r(x_0)}(x) \in W_0^{1,n}(\Om) \setminus L^\infty(\Om)$. In this borderline case, Trudinger \cite{Trudinger} and independently Pohozaev \cite{Pohozaev}, Yudovich \cite{Yudovich} show that $W^{1,n}_0(\Om) \hookrightarrow L_{\vphi_n}(\Om)$, where $L_{\varphi_n}(\Om)$ is the Orlicz space associated with the Young function $\varphi_n(t) = e^{c |t|^{n/(n-1)}} -1$ for some $c >0$. Later, Moser \cite{Moser} sharpened the Trudinger inequality by finding the optimal exponent $c$ as follows: 

\begin{Theorema}{\bf (The Moser--Trudinger inequality)}  Let $\Om$ be a bounded domain in $\R^n$, then it holds
\begin{equation}\label{eq:MT}
\sup_{u\in W_0^{1,n}(\Om), \|\na u\|_{L^n(\Om)} \leq 1} \int_\Om e^{\alpha_n |u|^{\frac n{n-1}}} dx < \infty.
\end{equation} 
where $\alpha_n = n \om_{n-1}^{\frac1{n-1}}$ and $\om_{n-1}$ denotes the surfaces area of the unit sphere in $\R^n$. Furthermore, the constant $\alpha_n$ is sharp in the sense that the supremum in \eqref{eq:MT} will be infinite if $\alpha_n$ is replaced by any larger number.
\end{Theorema}
The Moser--Trudinger inequality \eqref{eq:MT} was extended to higher order Sobolev spaces by Adams \cite{Adams} (which is now called Adams inequality). To state his inequality, we first fix some notation. For an integer $m \geq 1$ and a smooth function $u$, we use the notation
\[
\na^m u = 
\begin{cases}
\Delta^{\frac m2}u &\mbox{if $m$ is even,}\\
\na \Delta^{\frac{m-1}2} u &\mbox{if $m$ is odd.}
\end{cases}
\] 
We also denote by
\[
\al(n,m) =
\begin{cases}
\frac n{\om_{n-1}} \lt(\frac{\pi^{\frac n2} 2^m \Gamma(\frac m2)}{\Gamma(\frac{n-m}2)}\rt)^{\frac n{n-m}}&\mbox{if $m$ is even,}\\
\frac n{\om_{n-1}} \lt(\frac{\pi^{\frac n2} 2^m \Gamma(\frac {m+1}2)}{\Gamma(\frac{n-m+1}2)}\rt)^{\frac n{n-m}}&\mbox{if $m$ is odd.}
\end{cases}
\]
Evidently, we have $\al(n,1) = \al_n$. In \cite{Adams}, Adams proved the following:
 
\begin{Theoremb}{\bf (The Adams inequality)} Let $\Om$ be a bounded domain in $\R^n$ and $m$ be a positive integer less than $n$. Then, it holds
\begin{equation}\label{eq:Adams}
\sup_{u\in W_0^{n,m}(\Om), \|\na^m u\|_{L^{\frac nm}(\Om)} \leq 1} \int_{\Om} e^{\al(n,m) |u|^{\frac n{n-m}}} dx < \infty.
\end{equation}
Furthermore, the constant $\al(n,m)$ in \eqref{eq:Adams} is sharp in the sense that the supremum in \eqref{eq:Adams} will becomes infinite if $\al(n,m)$ is replaced by any larger number.
\end{Theoremb}

Both the Moser--Trudinger inequality and Adams inequality have many applications in Analysis, Geometry and Partial Differential Equations, especially in studying the curvature problems. There have been many generalizations of the Moser--Trudinger inequality and the Adams inequality in literature. For examples, the Moser--Trudinger inequality and Adams inequality have been generalized to unbounded domains and whole spaces in \cite{Adachi,Fontana,LamLu2012,LiRuf,VHN,Ruf,RufSani}, to Riemannian manifolds in \cite{Bertrand,YangSuKong,KaSan,Li2001,Li2005,ManciniSandeep}. The weighted Moser--Trudinger inequality in whole space $\R^n$ was established by Takahashi and the author in \cite{NguyenTakahashi}, while the singular version of the Moser--Trudinger inequality and the Adams inequality was proved by Adimurthi and Sandeep \cite{AdiSandeep}, by Adimurthi and Yang \cite{AdiYang}, and by Lam and Lu \cite{LamLu}. There also have been many improvements of the Moser--Trudinger inequality and the Adams inequality. The readers may consult these improvements in \cite{Lions,AdiDruet,LuYang2009,Yang,DelaTorre,LuYang2016,LuYang2017,LiLuYang,WangYe,Tintarev,VHN0,VHN1}. An interesting question related to the Moser--Trudinger inequality and the Adams inequality is whether or not the extremal functions exist? Concerning to this subject, we refer the readers to the papers \cite{CarlesonChang,Csato,Flucher,Li2005,Lin,LiRuf,Ruf,VHN,VHN1,ZhuChen} and references therein.

We next discuss about the Hardy and Rellich inequality. Let $\Om$ be a domain in $\R^n$, $n \geq 3$, the classical Hardy inequality says that
\begin{equation}\label{eq:Hardy}
I(\Om,u):=\int_{\Om} |\na u|^2 dx - \frac{(n-2)^2}4 \int_{\Om} \frac{|u|^2}{|x|^2}dx\geq 0,\quad\forall\, u\in C_0^\infty(\Om).
\end{equation}
The constant $(n-2)^2/4$ is sharp and never achieved. In the limiting case $n =2$, a non-trivial substitute of \eqref{eq:Hardy} is due to Leray \cite{Leray} who used it in the study of two dimensional viscous flows. More generally, it has been extended to $p=n \geq 2$ in \cite{AdiChauRa,AdiSandeep2002,Barbatis}: let $\Om$ be a bounded domain in $\R^n$ ($n\geq 2$) containing the origin and $R_\Om :=\sup_{x\in \Om} |x|$, then for any $u\in C_0^\infty(\Om)$ and $R \geq R_\Om$
\begin{equation}\label{eq:criticalHardy}
I_n[u,\Om,R]:=\int_\Om |\na u|^n dx -\lt(\frac{n-1}n\rt)^n \int_\Om \frac{|u|^n}{|x|^n E_1^n\lt(\frac{|x|}R\rt)} dx  \geq 0,  
\end{equation}
where $E_1(t) = \ln (e/t)$ for $t \in (0,1]$, and $((n-1)/n)^n$ is the best constant which is never achieved. It is an interesting question whether we have the Moser--Trudinger type inequality for functions satisfying the condition $I_n[u,\Om,R] \leq 1$. This question was firstly addressed by Psaradakis and Spector \cite{PS2015}. They show that there does not exist any positive constant $c>0$ for which the following inequality is true
\[
\sup_{u\in W_0^{1,n}(\Om), I_n[u,\Om,R_\Om] \leq 1}  \int_\Om e^{c |u|^{\frac n{n-1}}} dx < \infty.
\]
However, introducing a logarithmic factor, in the same paper, they established the following Lery--Trudinger inequality:
\begin{Theoremc} {\bf (Leray--Trudinger inequality)} Let $\Om$ be a bounded domain in $\R^n$, $n \geq 2$ containing the origin. For any $\ep >0$, there exist positive constants $A_{n,\ep}$ depending only on $n$ and $\ep$ and $B_n$ depending only on $n$ such that
\begin{equation}\label{eq:LerayTrudinger}
\sup_{u\in W_0^{1,n}(\Om), I_n[u,\Om,R_\Om] \leq 1}  \int_\Om e^{A_{n,\ep}\lt( \frac{|u|}{E_1^{\ep}\lt(\frac{|x|}{R_\Om}\rt)}\rt)^{\frac n{n-1}}} dx \leq B_n |\Om|.
\end{equation}
Furthermore, such an estimate fails for $\epsilon =0$. 
\end{Theoremc}

The Leray--Trudinger inequality \eqref{eq:LerayTrudinger} then was improved by Mallick and Tintarev \cite{MT2018} by showing that the inequality \eqref{eq:LerayTrudinger} still holds if we replace $E_1$ by $E_2(t):= \ln(e E_1(t))$ for $t \in (0,1]$ in the power of exponential. More precisely, they proved the following result: 
\begin{Theoremd}
{\bf (Improved Leray--Trudinger inequality)} Let $\Om$ be a bounded domain in $\R^n$, $n \geq 2$ containing the origin. For any $\beta \geq \frac 2n$ and $R \geq R_\Om$, there exist positive constants $A_{n,\ep}$ and $B_n$ depending only on $n$ such that for any $0< c < A_n$
\begin{equation}\label{eq:ImprovedLerayTrudinger}
\sup_{u\in W_0^{1,n}(\Om), I_n[u,\Om,R] \leq 1}  \int_\Om e^{c\lt( \frac{|u|}{E_2^{\beta}\lt(\frac{|x|}{R_\Om}\rt)}\rt)^{\frac n{n-1}}} dx \leq B_n |\Om|.
\end{equation}
Moreover, the supremum above is infinite if $\beta < \frac1n$ for any $c>0$. 
\end{Theoremd}
The situation is not clear when $\frac1n \leq \beta <\frac2n$. However, for $n=2$, when $\beta =\frac12$  the inequality \eqref{eq:ImprovedLerayTrudinger} is true when $\Om$ is the unit ball and $u$ is radial function (see Remark $1.2$ in \cite{MT2018}).

Let $m\geq 1$ be an integer less than $\frac n2$. The classical Rellich inequality in $\R^n$ (see \cite{Davies}) says that 
\begin{equation}\label{eq:Rellich}
\int_\Om |\na^m u|^2 dx - C_{n,m} \int_\Om \frac{|u|^2}{|x|^{2m}} dx \geq 0, \quad \forall\, u\in C_0^\infty(\Om),
\end{equation}
where the constant $C_{n,m}$ is given by
\[
C_{n,m} =
\begin{cases}
\frac{(n-2)^2}4 \prod_{k=0}^{\frac{m-1}2 -1} \frac{(n+2+4k)^2(n-2-4(k+1))^2}{16} &\mbox{if $m$ is odd,}\\
\prod_{k=0}^{\frac m2 -1} \frac{(n+4k)^2(n-4(k+1))^2}{16} &\mbox{if $m$ is even.}
\end{cases}
\]
Furthermore, the constant $C_{n,m}$ is sharp and never achieved. In the limiting case $n =2m$, we have an analogue of the Leray inequality for the higher order derivatives 
\begin{equation}\label{eq:higherLeray}
\tilde I_n[u,\Om,R]:= \int_\Om |\na^{\frac n2} u|^2 dx - \frac{(n-2)^2}4 C_{n,\frac n2 -2} \int_\Om \frac{|u|^2}{|x|^{n} E_1^2\lt(\frac{|x|}R\rt)} dx \geq 0, \quad \forall\, u\in C_0^\infty(\Om),
\end{equation}
In viewing of \eqref{eq:higherLeray} and the Leray--Trudinger inequality \eqref{eq:LerayTrudinger} and its improved version \eqref{eq:ImprovedLerayTrudinger}, we wonder whether we have the Adams type inequality under the condition $\tilde I_n[u,\Om,R]\leq 1$. We will address this question in this paper. Our first main result concerning to the dimension four reads as follows:
\begin{theorem}\label{Maintheorem}
Let $\Om$ be a bounded domain in $\R^4$ containing the origin. Then for any $\beta \geq 1$ and $R \geq R_\Om$, there exist positive constants $A$ and $B$ such that for any $c < A$,
\begin{equation}\label{eq:LerayAdams}
\sup_{u\in C_0^\infty(\Om), \tilde I_4[u,\Om,R] \leq 1}  \int_\Om e^{c\lt( \frac{|u|}{E_2^{\beta}\lt(\frac{|x|}R\rt)}\rt)^2} dx \leq B |\Om|.
\end{equation}
Furthermore, if $\beta < \frac12$ then the supremum in \eqref{eq:LerayAdams} is infinite for any $c>0$.
\end{theorem}

The situation is not clear when $\frac12 \leq \beta < 1$. However, if $\Om$ is the ball with center at the origin, the inequality \eqref{eq:LerayAdams} holds for $\beta =\frac12$ when restricting to radial function (see Theorem \ref{radial} below). As a consequence of Theorem \ref{Maintheorem}, we obtain the following extension of the Leray--Trudinger inequality \eqref{eq:LerayTrudinger} due to Psaradakis and Spector for the Laplacian operator in dimension four:
\begin{corollary}\label{HequaPS}
Let $\Om$ be a bounded domain in $\R^4$ containing the origin. Then for any $\ep >0$, there exist positive constants $\tilde A_{\ep}$ depending only on $\epsilon$ and $\tilde B$ such that
\begin{equation}\label{eq:LerayAdamsPS}
\sup_{u\in C_0^\infty(\Om), \tilde I_4[u,\Om,R] \leq 1}  \int_\Om e^{\tilde A_{\ep}\lt( \frac{|u|}{E_1^{\ep}\lt(\frac{|x|}{R_\Om}\rt)}\rt)^2} dx \leq \tilde B |\Om|.
\end{equation}
Furthermore, such an estimate fails for $\ep = 0$.
\end{corollary}

Let us make some comments on the proof of Theorem \ref{Maintheorem}. For the proof, we follow closely the Trudinger's original proof of the inequality \eqref{eq:MT} (see \cite{Trudinger}). This approach was used in \cite{MT2018} to prove the improved Leray--Trudinger inequality \eqref{eq:ImprovedLerayTrudinger} (see also \cite{PS2015} for the proof of the Leray--Trudinger inequality \eqref{eq:LerayTrudinger}). More precisely, they performed the following ground state transform $u(x) = E_1^{\frac{n-1}n}\lt(\frac{|x|}R\rt) v(x)$ and obtained the following estimate
\begin{equation}\label{eq:GStransform}
I_n[u,\Om,R] \geq C_1(n) \int_\Om |\na v|^n E_1^{n-1}\lt(\frac{|x|}R\rt) dx, \quad u\in C_0^\infty(\Om),
\end{equation}
for any $R \geq R_\Om$ with $C_1(n) = (2^{n-1} -1)^{-1}$. Then by using the following expression of $v$ via its gradient
\begin{equation}\label{eq:expression}
v(x) = \frac1{\om_{n-1}} \int_\Om \frac{\la x- y, \na v(y)\ra}{|x-y|^n} dy,
\end{equation}
and integral estimations, they get the following estimates
\begin{equation}\label{eq:qnormbound}
\lt(\int_\Om \lt|\frac{u(x)}{E_2^{\frac2n}\lt(\frac{|x|}R\rt)}\rt|^q dx\rt)^\frac1q \leq C(n) \lt(1 + \frac{q(n-1)}n\rt)^{1-\frac1n + \frac1q} |\Om|^{\frac1q} (I_n[u,\Om,R])^{\frac1n}
\end{equation}
for some constant $C(n)$ depending only on $n$ (see Proposition $3.1$ in \cite{MT2018}). The inequality \eqref{eq:ImprovedLerayTrudinger} follows from the estimate \eqref{eq:qnormbound}. Now, following \cite{MT2018} we also make a ground state transform as $u(x) = E_1^{\frac12}\lt(\frac{|x|}R\rt) v(x)$. We will establish the following estimate (see Proposition \ref{GSrepresentation} below) which is an analogue of \eqref{eq:GStransform}
\begin{equation}\label{eq:GS2}
\tilde I_4[u,\Om,R] = \int_\Om (\De u)^2 dx -\int_\Om \frac{u^2}{|x|^4 E_1^2\lt(\frac{|x|}R\rt)} dx \geq \frac 23
 \int_\Om E_1\lt(\frac{|x|}R\rt) (\Delta v)^2 dx.
\end{equation}
Instead of \eqref{eq:expression}, we use the following expression of $v$ via its Laplacian (see \cite{Adams})
\begin{equation}\label{eq:expression2}
v(x) = -\frac1{4\pi^2}\int_{\Om} \frac{\De v(y)}{|x-y|^{2}} dy,\quad v \in C_0^\infty(\Om).
\end{equation}
To finish the proof of Theorem \ref{Maintheorem}, we shall establish the estimates like \eqref{eq:qnormbound} by using the estimate \eqref{eq:GS2} and \eqref{eq:expression2}. 

We wonder if the Theorem \ref{Maintheorem} still holds in higher dimensions or holds for higher order derivatives. It seems that this problem is not easy to handle. Indeed, following the original approach of Moser, the main difficulty is to establish an analogue of the estimate \eqref{eq:GS2} in higher dimensions or in higher order derivatives. However, if we restrict ourselves to the radial case, we can prove the Leray--Adams inequality in higher order derivatives. To state our next result, let us recall a Hardy--Rellich type inequality for radial functions. Let $p>1$, and $n, m$ be integers such that $1 \leq m < n/p$ and $B_r$ be the ball of radius $r >0$ with the center at the origin in $\R^n$. Then the following inequality holds
\begin{equation}\label{eq:Rellichradial}
\int_{B_r} |\na^m u|^p dx \geq R_{n,m,p}\int_{B_r} \frac{|u|^p}{|x|^{mp}} dx
\end{equation}
for any radial function $u \in C_0^\infty(B_r)$, with 
\[
R_{n,m,p} =
\begin{cases}
\prod_{j=0}^{k-1} \lt(\frac{n(p-1) + 2jp}{p} \frac{n-2(j+1)p}p\rt)^p &\mbox{if $m =2k$}\\
\lt(\frac{n-p}p\rt)^p \prod_{j=0}^{k-1} \lt(\frac{n(p-1) +p+ 2jp}{p} \frac{n-p-2(j+1)p}p\rt)^p &\mbox{if $m =2k+ 1$}.
\end{cases}
\]
We refer the readers to \cite{Davies} for the proof of \eqref{eq:Rellichradial}. In fact, this inequality holds true for any function in $C_0^\infty(\R^n)$ without the radiality assumption. In the critical case $pm =n$, we have the following inequality
\begin{equation}\label{eq:criticalRellichradial}
I_{n,m}[u,B_r,R] := \int_{B_r} |\na^m u|^{\frac nm} dx - R_{n,m} \int_{B_r} \frac{|u|^{\frac nm}}{|x|^n E_1^{\frac nm} \lt(\frac{|x|}R\rt)} dx \geq 0,
\end{equation}
for any radial function $u \in C_0^\infty(B_r)$ with 
\[
R_{n,m} = \lt(\frac{(n-2)(n-m)}n\rt)^{\frac nm}R_{n,m-2,\frac nm},
\]
with convention that $R_{n,0,p} =1$ for $p>1$. The inequality \eqref{eq:criticalRellichradial} for $m=2$ was proved in \cite{AdiSan}. In fact, in that paper, Adimurthi and Santra proved the inequality \eqref{eq:criticalRellichradial} for $m=2$ without radiality assumption. For the convenience of readers, we give the proof of \eqref{eq:criticalRellichradial} for any order below. The next theorem provides an extension of Theorem \ref{Maintheorem} to higher dimension or higher order derivatives in the radial case. More precisely, we prove the following result:
\begin{theorem}\label{radial}
Let $m$ and $n$ be the integers such that $2\leq m \leq \frac{n}2$. Let us denote by $B_r$ the ball with center at the origin and radius $r >0$ in $\R^n$, and by $C_{0,rad}^\infty(B_r)$  the space of radial functions in $C_0^\infty(B_r)$. Then there exist the constant $a_{n,m}$ and $b_{n,m}$ depending only on $n$ and $m$ such that for any $\beta \geq \frac{n-m}n$, $R\geq r$ and $c \leq a_{n,m}$
\begin{equation}\label{eq:radialcasem}
\sup_{u\in C_{0,rad}^\infty(B_r), I_{n,m}[u,B_r,R] \leq 1}\int_{B_r} e^{c\lt(\frac{|u|}{E_2^\beta\lt(\frac{|x|}R\rt)}\rt)^{\frac{n}{n-m}}} dx \leq b_{n,m} |B_r|.
\end{equation}

\end{theorem}
Our approach to Theorem \ref{radial} is completely different with the one to Theorem \ref{Maintheorem}. Instead of using the estimate for the ground state transformation, we will use the Hardy--Rellich type inequality to reduce the order of derivative to one. Next, we exploit the radiality to estimate the decay of functions via its gradient. Theorem \ref{radial} then follows from this estimate. The detail proof of Theorem \ref{radial} will be given in Section \S5 below. We should emphasize here that the approach to prove the Adams inequality via the sharp Hardy--Rellich inequality was used by the author to give a new proof of the the Hardy--Adams inequality due to Lu and Yang \cite{LuYang2017} and Li, Lu and Yang \cite{LiLuYang} in \cite{VHN2}, and to establish the sharp Adams inequality in the fractional Sobolev--Slobodeckij spaces \cite{VHN3}.

The rest of this paper is organized as follows. In the next section \S2, we prove several results concerning to the Hardy--Rellich type inequality in dimension four which will be used in the proof of Theorem \ref{Maintheorem}. In Section \S3, we use the results in Section \S2 to estimate the $L^q$ norm of $E_2^{-1}(|x|/R) u(x)$ which plays an important role in the proof of Theorem \ref{Maintheorem} following Trudinger's original approach. Section \S5 is devoted to prove Theorem \ref{radial}. 

\section{Preliminaries}
In this sections, we are going to recall and prove some useful results which we will use the proof of Theorem \ref{Maintheorem}. But before that let us fix the notations for the rest of this paper. For a bounded domain $\Om$ in $\R^n$, we set $R_\Om = \sup_{x\in \Om} |x|$. We define $E_1(t) = 1 -\ln t$ and $E_2(t) = \ln(e E_1(t))$ for $t \in (0,1]$. We denote by $B_R$ the ball with center at the origin and radius $R >0$. For simplicity, we denote $B_1$ by $B$.

The next proposition gives us an analogue of \eqref{eq:GStransform} for the Laplacian operator. Its proof uses the spherical decomposition and the one dimensional Hardy inequality.

\begin{proposition}\label{GSrepresentation}
Let $\Om$ be a bounded domain in $\R^4$ containing the origin. Then for any function $u \in C_0^\infty(\Om\setminus\{0\})$ it holds
\begin{equation}\label{eq:GSrepresentation}
\int_\Om (\De u)^2 dx -\int_\Om \frac{u^2}{|x|^4 E_1^2\lt(\frac{|x|}R\rt)} dx \geq \frac 23
 \int_\Om E_1\lt(\frac{|x|}R\rt) (\Delta v)^2 dx
\end{equation}
with $v(x) = E_1^{-\frac12}\lt(\frac{|x|}R\rt) u(x)$.
\end{proposition}
\begin{proof}
We first remark that if we denote $\tilde \Om = R_\Om^{-1} \Om$ and $\tilde R = R/ R_\Om \geq 1$. We have $\tilde \Om \subset B$ and $R_{\tilde \Om} =1$. Given a function $u \in C_0^\infty(\Om \setminus\{0\})$, define $\tilde u(x) = u (R_\Om x)$. Then $\tilde u \in C_0^\infty(\tilde \Om \setminus\{0\})$ and $\tilde v(x) =E_1^{-\frac12}(|x|/\tilde R) \tilde u(x) = v (R_\Om x)$. We have
\[
\int_{\tilde \Om} (\De \tilde u)^2 dx = \int_{\Om} (\De u)^2 dx,
\]
\[
\int_{\tilde \Om} \frac{\tilde u^2}{|x|^4 E_1^2\lt(\frac{|x|}{\tilde R}\rt)} dx =\int_\Om \frac{u^2}{|x|^4 E_1^2\lt(\frac{|x|}R\rt)} dx,
\]
and
\[
\int_{\tilde \Om} E_1\lt(\frac{|x|}{\tilde R}\rt) (\Delta \tilde v)^2 dx = \int_\Om E_1\lt(\frac{|x|}R\rt) (\Delta v)^2 dx.
\]
Then, without loss of generality, we can assume $\Om \subset B$ and $R_\Om =1$, and consider $R\geq 1$.

We first consider the case that $\Om$ is the unit ball $B$ and $u$ is radial function. For convenience, we write $u(x) = u(|x|)$ for a radial function $u$ , then $\De u(x) = u''(r) +3 u'(r)/r$ with $r =|x|$.  Define $w(t) =  u(e^{-t})$ with $t \geq 0$. Using the polar coordinate and making the change of variable $r =e^{-t}$ we have
\begin{align*}
\int_{B} (\De u)^2 dx & = 2\pi^2\int_0^1 ( u''(r) + 3 r^{-1}  u'(r))^2 r^3 dr\\
& =2 \pi^2 \int_0^\infty (w''(t) -2 w'(t))^2 dt\\
& =2 \pi^2 \int_0^\infty (w''(t))^2 dt + 8\pi^2 \int_0^\infty (w'(t))^2 dt,
\end{align*}
here we use integration by parts,
\[
\int_\Om \frac{u^2}{|x|^4 E_1^2(|x|/R)} dx = 2\pi^2 \int_0^\infty \frac{w(t)^2}{(a+ t)^2} dt,
\]
with $a = 1 + \ln R \geq 1$. Thus, we have
\begin{align}\label{eq:exp}
\int_{B} (\De u)^2 dx -&\int_\Om \frac{u^2}{|x|^4 E_1^2(|x|/R)} dx \notag\\
&= 2\pi^2\lt(\int_0^\infty (w''(t))^2 dt + 4\int_0^\infty (w'(t))^2 dt -\int_0^\infty \frac{w(t)^2}{(a+ t)^2} dt\rt).
\end{align}
Let $\bar w(t) = w(t)/(t+ a)^{\frac12} = v(e^{-t})$. By using polar coordinate and making the change of variable $r =e^{-t}$, we have
\begin{align}\label{eq:exp1}
\int_B E_1\lt(\frac{|x|}R\rt) (\De v)^2 &=2\pi^2 \int_0^1 E_1(r/R) (v''(r) + 3 r^{-1} v(r))^2 r^3 dr\notag\\
&= 2\pi^2 \int_0^\infty (a+ t) (\bar w''(t) -2 \bar w'(t))^2 dt\notag\\
&= 2\pi^2\Bigg(\int_0^\infty (a+t)(\bar w''(t))^2 dt + 4\int_0^\infty (a+t)(\bar w'(t))^2 dt \notag\\
&\quad\qquad \qquad+ 2 \int_0^\infty (\bar w'(t))^2 dt\Bigg),
\end{align}
here we use integration by parts. Now, by direct computations, we have
\[
\bar w'(t) = \frac{w'(t)}{(a+ t)^{\frac12}} - \frac12 \frac{w(t)}{(a+t)^{\frac32}},
\]
and
\[
\bar w''(t) = \frac{w''(t)}{(a+ t)^{\frac12}} -  \frac{w'(t)}{(a+t)^{\frac32}} + \frac34 \frac{w(t)}{(a+t)^{\frac52}}.
\]
Therefore, using integration by parts, we get
\begin{equation}\label{eq:IBP1}
\int_0^\infty (\bar w'(t))^2 dt = \int_0^\infty \frac{(w'(t))^2}{(a+t)} dt - \frac34 \int_0^\infty \frac{w(t)^2}{(a+t)^3} dt,
\end{equation}
\begin{equation}\label{eq:IBP2}
\int_0^\infty (a+t)(\bar w'(t))^2 dt = \int_0^\infty (w'(t))^2 dt -\frac14 \int_0^\infty \frac{w(t)^2}{(a+t)^2} dt,
\end{equation}
and 
\begin{align}\label{eq:IBP3}
\int_0^\infty (a+t)(\bar w''(t))^2 dt& = \int_0^\infty (w''(t))^2 dt + \int_0^\infty \frac{(w'(t))^2}{(a+t)^2} dt + \frac9{16} \int_0^\infty \frac{w(t)^2}{(a+t)^4} dt\notag\\
&\, -2 \int_0^\infty \frac{w''(t) w'(t)}{a+t} dt -\frac32 \int_0^\infty \frac{w'(t) w(t)}{(a+t)^3} dt + \frac32 \int_0^\infty \frac{w''(t) w(t)}{(a+t)^2} dt \notag\\
&=\int_0^\infty (w''(t))^2 dt - \frac32 \int_0^\infty \frac{(w'(t))^2}{(a+t)^2} dt - \frac{27}{16} \int_0^\infty \frac{w(t)^2}{(a+t)^4} dt\notag\\
&\,+ 3 \int_0^\infty \frac{w'(t) w(t)}{(a+t)^3} dt\notag\\
&=\int_0^\infty (w''(t))^2 dt -\frac32 \int_0^\infty \frac{(w'(t))^2}{(a+t)^2} dt +\frac{45}{16} \int_0^\infty \frac{w(t)^2}{(a+t)^4} dt.
\end{align}
Inserting \eqref{eq:IBP1}, \eqref{eq:IBP2} and \eqref{eq:IBP3} into \eqref{eq:exp1} yields
\begin{align}\label{eq:exp2}
\int_B E_1\lt(\frac{|x|}R\rt) (\De v)^2dx & =2 \pi^2 \Bigg(\int_0^\infty (w''(t))^2 dt + 4\int_0^\infty (w'(t))^2 dt - \int_0^\infty \frac{w(t)^2}{(a+t)^2} dt\notag\\
&\qquad\qquad\qquad + 2\int_0^\infty \frac{(w'(t))^2}{(a+t)} dt - \frac32 \int_0^\infty \frac{w(t)^2}{(a+t)^3} dt\notag\\
&\qquad\qquad\qquad -\frac32 \int_0^\infty \frac{(w'(t))^2}{(a+t)^2} dt +\frac{45}{16} \int_0^\infty \frac{w(t)^2}{(a+t)^4} dt\Bigg).
\end{align}
By Hardy inequality, we have
\begin{equation}\label{eq:H1}
\int_0^\infty \frac{(w'(t))^2}{(a+t)^2} dt \geq \frac{9}{4} \int_0^\infty \frac{w(t)^2}{(a+t)^4} dt.
\end{equation}
It follows from \eqref{eq:IBP1} and \eqref{eq:IBP2} that
\begin{equation}\label{eq:compare1}
2\int_0^\infty \frac{(w'(t))^2}{(a+t)} dt - \frac32 \int_0^\infty \frac{w(t)^2}{(a+t)^3} dt \leq \frac2a\lt(\int_0^\infty (w'(t))^2 dt -\frac14 \int_0^\infty \frac{w(t)^2}{(a+t)^2} dt\rt).
\end{equation}
Plugging \eqref{eq:H1} and \eqref{eq:compare1} into \eqref{eq:exp2} we have
\begin{align}\label{eq:exp3}
\int_B E_1\lt(\frac{|x|}R\rt) (\De v)^2 dx& =2 \pi^2 \int_0^\infty (w''(t))^2 dt \notag\\
&\qquad + \lt(1 +\frac1 {2a}\rt)2\pi^2 \lt(4\int_0^\infty (w'(t))^2 dt - \int_0^\infty \frac{w(t)^2}{(a+t)^2} dt\rt).
\end{align}
Comparing \eqref{eq:exp} and \eqref{eq:exp3}, we arrive the following estimate
\begin{equation}\label{eq:radialcase}
\int_{B} (\De u)^2 dx -\int_\Om \frac{u^2}{|x|^4 E_1^2(|x|/R)} dx \geq \frac{2a}{2a+1} \int_B E_1\lt(\frac{|x|}R\rt) (\De v)^2 dx.
\end{equation}
Since $a = 1 + \ln R \geq 1$, then $\frac{2a}{2a+1} \geq \frac 23$.

We next consider the general case. For a function $u\in C_0^\infty(\Om\setminus\{0\})$, we can decompose it as
\begin{equation}\label{eq:harmonicexpansion1}
u(x) = \sum_{k=0}^\infty u_k(r) \phi_k(\eta),\quad\text{\rm with}\quad x = r \eta, \, \eta \in S^{n-1}, \, r = |x|,
\end{equation}
 where $\phi_k$ is the eigenfunction of the Laplace--Beltrami operator on the unit sphere $S^3$ with respect to the eigenvalue $c_k = -k(2+k), k \geq 0$ and $\int_{S^3} |\phi_k(\eta)|^2 d\eta =1$. Notice that $u_k \in C_0^\infty(B\setminus\{0\})$ is radial function and 
\[
\De u(x) = \sum_{k=0}^\infty\lt(\De u_k(r) - c_k \frac{u_k(r)}{r^2}\rt) \phi_k(\eta), \quad x = r\eta, \, \eta \in S^3.
\]
Using the fact $\int_{S^3} \phi_k(\eta) \phi_l(\eta) d\eta = \de_{kl}$ and the polar coordinate, we have
\begin{equation}\label{eq:tichphanu1}
\int_{\Om} \frac{u^2}{|x|^n E_1^2\lt(\frac{|x|}R\rt)} dx = \sum_{k=0}^\infty \int_{B} \frac{u_k^2}{|x|^4 E_1^2\lt(\frac{|x|}R\rt)} dx,
\end{equation}
and 
\begin{align}\label{eq:De11}
\int_{\Om}(\De u)^2 dx &= \sum_{k=0}^\infty \int_{B} \lt(\Delta u_k - c_k \frac{u_k}{|x|^2}\rt)^2 dx \notag\\
&=\sum_{k=0}^\infty\lt(\int_{B} (\De u_k)^2 dx -2c_k \int_{B}\frac{\De u_k \, u_k}{|x|^2} dx + c_k^2 \int_{B} \frac{u_k^2}{|x|^4} dx\rt),
\end{align}
here we used integration by parts. In other hand, using $\De(u_k^2) = 2 u_k \De u_k + 2 |\na u_k|^2$ and the integration by parts, we get
\begin{equation}\label{eq:De21}
2\int_{B} \frac{\De u_k \, u_k}{|x|^{2}} dx= -2 \int_{B} \frac{|\na u_k|^2}{|x|^{2}} dx.
\end{equation} It follows from \eqref{eq:tichphanu1}, \eqref{eq:De11} and \eqref{eq:De21} that
\begin{align}\label{eq:ex11}
\int_{\Om}(\De u)^2 dx - \int_{\Om} \frac{u^2}{|x|^4 E_1^2\lt(\frac{|x|}R\rt)} dx&=  \sum_{k=0}^\infty \lt(\int_{B} (\De u_k)^2dx - \int_{B} \frac{u_k^2}{|x|^4 E_1^2\lt(\frac{|x|}R\rt)} dx\rt) \notag\\
&\,  + 2\sum_{k=0}^\infty c_k \int_{B} \frac{|\na u_k|^2}{|x|^2} dx + \sum_{k=0}^\infty c_k^2 \int_{B} \frac{u_k^2}{|x|^4} dx.
\end{align}
Notice that
\[
v(x) = \sum_{k=0}^\infty v_k(r) \phi_k(\eta),\quad\text{\rm with} \quad v_k(x)= E_1^{-\frac12}(|x|/R) u_k(x).
\]
We have
\begin{align*}
\int_\Om E_1\lt(\frac{|x|}R\rt) (\De v)^2 dx&= \sum_{k=0}^\infty \int_{B}E_1\lt(\frac{|x|}R\rt) \lt(\De v_k -c_k \frac{v_k}{|x|^2}\rt)^2 dx\\
&= \sum_{k=0}^\infty  \int_{B} E_1\lt(\frac{|x|}R\rt) (\De v_k)^2 dx + \sum_{k=0}^\infty c_k^2 \int_{B} E_1\lt(\frac{|x|}R\rt) \frac{v_k^2}{|x|^4} dx\\
&\quad - \sum_{k=0}^\infty 2c_k\int_{B} E_1\lt(\frac{|x|}R\rt) \frac{\De v_k\, v_k}{|x|^2} dx.
\end{align*}
By integration by parts, we have
\begin{align*}
\int_{B} E_1\lt(\frac{|x|}R\rt) \frac{\De v_k\, v_k}{|x|^2} dx&= \frac12 \int_{B} \De (v_k^2) \frac{E_1(|x|/R)}{|x|^2} dx -\int_{B} E_1\lt(\frac{|x|}R\rt) \frac{|\na v_k|^2}{|x|^2} dx\\
&=\int_{B} \frac{v_k^2}{|x|^4} dx -\int_{B} E_1\lt(\frac{|x|}R\rt) \frac{|\na v_k|^2}{|x|^2} dx.
\end{align*}
So, we now have 
\begin{align}\label{eq:ss1}
\int_\Om E_1\lt(\frac{|x|}R\rt) (\De v)^2 dx&=\sum_{k=0}^\infty  \int_{B} E_1\lt(\frac{|x|}R\rt) (\De v_k)^2 dx + \sum_{k=0}^\infty c_k^2 \int_{B} E_1\lt(\frac{|x|}R\rt) \frac{v_k^2}{|x|^4} dx\notag\\
&\quad - \sum_{k=0}^\infty 2c_k\int_{B} \frac{v_k^2}{|x|^4} dx +\sum_{k=0}^\infty 2c_k \int_{B} E_1\lt(\frac{|x|}R\rt) \frac{|\na v_k|^2}{|x|^2} dx\notag\\
&\leq \sum_{k=0}^\infty  \int_{B} E_1\lt(\frac{|x|}R\rt) (\De v_k)^2 dx + \sum_{k=0}^\infty c_k^2 \int_{B} E_1\lt(\frac{|x|}R\rt) \frac{v_k^2}{|x|^4} dx\notag\\
&\quad +\sum_{k=0}^\infty 2c_k \int_{B} E_1\lt(\frac{|x|}R\rt) \frac{|\na v_k|^2}{|x|^2} dx.
\end{align}
It can be easily shown by integration by parts that
\begin{equation}\label{eq:na1}
\int_{B} \frac{|\na u_k|^2}{|x|^2} dx = \frac14 \int_{B} \frac{v_k^2}{|x|^4 E_1\lt(\frac{|x|}R\rt)} dx + \int_{B} E_1\lt(\frac{|x|}R\rt) \frac{|\na v_k|^2}{|x|^2} dx.
\end{equation}
Indeed, we have 
\[
\na u_k(x) = -\frac12 E_1^{-\frac12}(|x|/R) \frac{x}{|x|^2} \tilde u_k(x) + E_1^{\frac12}(|x|/R) \na v_k(x).
\]
Hence, integrating the expansion of $\frac{|\na u_k|^2}{|x|^2}$ and using integration by parts, we get the equality \eqref{eq:na1}.

Inserting \eqref{eq:na1} into \eqref{eq:ex11}, we get
\begin{align*}
\int_{\Om}(\De u)^2 dx -& \int_{\Om} \frac{u^2}{|x|^4 E_1^2\lt(\frac{|x|}R\rt)} dx\\
&=  \sum_{k=0}^\infty \lt(\int_{B} (\De u_k)^2dx - \int_{B_{R_1}} \frac{u_k^2}{|x|^4 E_1^2\lt(\frac{|x|}R\rt)} dx\rt) \notag\\
&\,  + 2\sum_{k=0}^\infty c_k \int_{B_{R_1}}  E_1\lt(\frac{|x|}R\rt)\frac{|\na v_k|^2}{|x|^2} dx +\frac12 \sum_{k=0}^\infty c_k\int_{B} \frac{v_k^2}{|x|^4 E_1\lt(\frac{|x|}R\rt)} dx\notag\\
&\qquad+  \sum_{k=0}^\infty c_k^2 \int_{B_{R_1}} E_1\lt(\frac{|x|}R\rt) \frac{v_k^2}{|x|^4} dx.
\end{align*}
Applying the inequality \eqref{eq:radialcase} for radial functions, we obtain
\begin{align}\label{eq:ex111}
\int_{\Om}(\De u)^2 dx -& \int_{\Om} \frac{u^2}{|x|^4 E_1^2\lt(\frac{|x|}R\rt)} dx\notag\\
&\geq  \frac23\sum_{k=0}^\infty \int_{B} E_1\lt(\frac{|x|}R\rt) (\De v_k)^2 dx  + 2\sum_{k=0}^\infty c_k \int_{B_{R_1}}  E_1\lt(\frac{|x|}R\rt)\frac{|\na v_k|^2}{|x|^2} dx\notag\\
&\qquad+  \sum_{k=0}^\infty c_k^2 \int_{B_{R_1}} E_1\lt(\frac{|x|}R\rt) \frac{v_k^2}{|x|^4} dx.
\end{align}
Comparing \eqref{eq:ss1} and \eqref{eq:ex111}, we obtain the desired inequality \eqref{eq:GSrepresentation}.
\end{proof}

We also need the following result.
\begin{proposition}\label{wcRellich}
Let $\Om$ be a bounded domain in $\R^4$ containing the origin. Suppose $u \in C_0^\infty(\Om\setminus\{0\})$ and $R \geq \R_\Om$, denote $v(x) = E_1^{-\frac12}\lt(\frac{|x|}R\rt) u(x)$. Then we have
\begin{equation}\label{eq:wcRellich}
\int_\Om (\De u)^2 dx -  \int_{\Om} \frac{|u|^2}{|x|^4 E_1^2\lt(\frac{|x|}R\rt)} dx \geq 3\int_{\Om} E_1\lt(\frac{|x|}R\rt)\frac{ |\na v|^2}{|x|^2} dx.
\end{equation}
\end{proposition}
\begin{proof}
By scaling argument, we can assume that $\Om \subset B$ and $R_\Om =1$.

We first show that
\begin{equation}\label{eq:Rellichradialgradient}
\int_{B}(\De w)^2 dx \geq  4 \int_{B} \frac{|\na w|^2}{|x|^2} dx,
\end{equation}
for any radial function $w \in C_0^\infty(B\setminus\{0\})$. Indeed, we have
\begin{align*}
\int_{B} (\De w)^2 dx& = 2\pi^2\int_0^\infty \lt(w''(r) + \frac{3} r w'(r)\rt)^2 r^3 dr\\
&= 2\pi^2\int_0^\infty (w''(r))^2 r^3 dr + 6\pi^2 \int_0^\infty [(w'(r))^2]' r^2 dr + 18 \pi^2 \int_0^\infty (w'(r))^2 r dr,
\end{align*} 
here by abusing the notation, we write $w(x) = w(r)$ with $r =|x|$. Applying the Hardy inequality for the first term, and using the integration by parts for the second term in the right hand side of the previous equality, we obtain \eqref{eq:Rellichradialgradient}.

We next show that
\begin{equation}\label{eq:wcRellichradialgradient}
\int_{B}(\De w)^2 dx -  \int_{B} \frac{|w|^2}{|x|^4 E_1^2\lt(\frac{|x|}R\rt)} dx \geq 4\int_{B} E_1\lt(\frac{|x|}R\rt) \frac{|\na \tilde w(x)|^2}{|x|^2} dx,
\end{equation}
for any radial function $w \in C_0^\infty(B\setminus\{0\})$, with $\tilde w(x) = E_1^{-\frac12}(|x|/R) w(x)$. Indeed, \eqref{eq:wcRellichradialgradient} follows from \eqref{eq:Rellichradialgradient} and \eqref{eq:na1} applied to $w$ instead of $u_k$.

In general, for any function $u \in C_0^\infty(\Om\setminus\{0\})$ we can decompose it as \eqref{eq:harmonicexpansion1}. Let $v_k(x) = E_1^{-\frac12}(|x|/R) u_k(x)$, then we have
\[
v(x) = \sum_{k=0}^\infty v_k(r) \phi_k(\eta),\quad\text{\rm with} \quad v_k(x)= E_1^{-\frac12}(|x|/R) u_k(x).
\]
and
\begin{align}\label{eq:na2}
\int_{\Om} E_1\lt(\frac{|x|}R\rt) \frac{|\na v|^2}{|x|^2} dx &=\sum_{k=0}^\infty \int_{B} E_1\lt(\frac{|x|}R\rt) \frac{|\na v_k|^2}{|x|^2} dx + \sum_{k=0}^\infty c_k \int_{B} E_1\lt(\frac{|x|}R\rt)\frac{v_k^2}{|x|^4} dx\notag \\
&= \sum_{k=0}^\infty \int_{B_{R_\Om}} E_1\lt(\frac{|x|}R\rt) \frac{|\na v_k|^2}{|x|^2} dx + \sum_{k=0}^\infty c_k \int_{B_{R_\Om}} \frac{u_k^2}{|x|^4} dx.
\end{align} It follows from \eqref{eq:ex11}, \eqref{eq:wcRellichradialgradient}, \eqref{eq:na1} and \eqref{eq:na2} that
\begin{align*}
\int_{\Om}(\De u)^2 dx - \int_{\Om} \frac{u^2}{|x|^4 E_1^2\lt(\frac{|x|}R\rt)} dx&\geq  \sum_{k=0}^\infty (4 +2c_k)\int_B E_1\lt(\frac{|x|}R\rt) \frac{|\na v_k|^2}{|x|^2} dx  + \sum_{k=0}^\infty c_k^2 \int_{B} \frac{u_k^2}{|x|^4} dx\\
&\geq 3\sum_{k=0}^\infty \lt( \int_{B_{R_\Om}} E_1\lt(\frac{|x|}R\rt) \frac{|\na v_k|^2}{|x|^2} dx +  c_k \int_{B_{R_\Om}} \frac{u_k^2}{|x|^4} dx\rt)\\
&= 3 \int_{\Om} E_1\lt(\frac{|x|}R\rt) \frac{|\na v|^2}{|x|^2} dx,
\end{align*}
here we use $c_0 =0$, $c_1=3$ and $c_k \geq 8$ for $k \geq 2$. The proof is completed.
\end{proof}

The next proposition give us a critical Rellich inequality in dimension four with the remainder term. We also show that the remainder term is sharp. This is an extension of Theorem B and Proposition $3.2$ in \cite{BFT} to the case of Laplacian operator.

\begin{proposition}\label{criticalRellich}
Let $\Om$ be a bounded domain in $\R^4$ containing the origin. Then for any $R \geq R_\Om$ and all function $u \in C_0^\infty(\Om \setminus\{0\})$, there holds
\begin{align}\label{eq:criticalRellich}
\int_\Om (\De u)^2 dx - \int_\Om \frac{|u|^2}{|x|^4 E_1^2\lt(\frac{|x|}R\rt)} dx \geq \int_\Om \frac{|u|^2}{|x|^4 E_1^2\lt(\frac{|x|}R\rt) E_2^2\lt(\frac{|x|}R\rt)} dx.
\end{align}
Furthermore, if there exists a positive constant $D >0$ for which the following inequality holds true
\begin{align}\label{eq:bestexponent}
\int_\Om (\De u)^2 dx - \int_\Om \frac{|u|^2}{|x|^4 E_1^2\lt(\frac{|x|}R\rt)} dx \geq D\int_\Om \frac{|u|^2}{|x|^4 E_1^2\lt(\frac{|x|}R\rt) E_2^\gamma \lt(\frac{|x|}R\rt)} dx
\end{align}
for all $u\in C_0^\infty(\Om \setminus\{0\})$ and for some $\gamma \in \R$, then 
\begin{itemize}
\item $\gamma \geq 2$,
\item and if $\gamma =2$ then $D \leq 1$.
\end{itemize}
Therefore, $1$ is the best constant in the right hand side of \eqref{eq:criticalRellich}.
\end{proposition}

\begin{proof}
By rescaling argument, we can assume that $R_\Om =1$, and consider $R \geq 1$.

We first prove \eqref{eq:criticalRellich}. Using again the decomposition \eqref{eq:harmonicexpansion1} and the formulas \eqref{eq:De11} and \eqref{eq:De21}, we have
\begin{equation*}
\int_{\Om}(\De u)^2 dx =\sum_{k=0}^\infty\lt(\int_{B} (\De u_k)^2 dx +2c_k \int_{B}\frac{|\na u_k|^2}{|x|^2} dx + c_k^2 \int_{B} \frac{u_k^2}{|x|^4} dx\rt).
\end{equation*}
Since $\int_{S^3} \phi_k(\eta) \phi_l(\eta) d\eta = \de_{kl}$, we then have
\begin{equation*}
\int_\Om \frac{|u|^2}{|x|^4 E_1^2\lt(\frac{|x|}R\rt)} dx = \sum_{k=0}^\infty \int_{B} \frac{|u_k|^2}{|x|^4 E_1^2\lt(\frac{|x|}R\rt)} dx,
\end{equation*}
and 
\begin{equation*}
\int_\Om \frac{|u|^2}{|x|^4 E_1^2\lt(\frac{|x|}R\rt) E_2^2 \lt(\frac{|x|}R\rt)} dx=\sum_{k=0}^\infty \int_{B} \frac{|u_k|^2}{|x|^4 E_1^2\lt(\frac{|x|}R\rt) E_2^2 \lt(\frac{|x|}R\rt)} dx.
\end{equation*}
So, it is enough to show that
\begin{equation}\label{eq:dkdu}
\int_{B} (\De u_k)^2 dx -\int_{B} \frac{|u_k|^2}{|x|^4 E_1^2\lt(\frac{|x|}R\rt)} dx \geq \int_{B} \frac{|u_k|^2}{|x|^4 E_1^2\lt(\frac{|x|}R\rt) E_2^2 \lt(\frac{|x|}R\rt)} dx.
\end{equation}
Note that $u_k$ is radial function, then the following Rellich type inequality holds
\begin{equation}\label{eq:Rellichgradient}
\int_{B} (\De u_k)^2 dx \geq 4 \int_\Om \frac{|\na u_k|^2}{|x|^2} dx.
\end{equation}
Following the proof of Theorem B in \cite{BFT}, let us define the vector field
\[
T(x) = \frac12 \frac{x}{|x|^4} E_1^{-1}\lt(\frac{|x|}R\rt)\lt( 1 + E_2^{-1}\lt(\frac{|x|}R\rt)\rt).
\]
It is easy to check that
\begin{equation}\label{eq:daoham}
E_1'(t) = -\frac1t,\quad \text{and}\quad E_2'(t) = -\frac1{t E_1(t)}.
\end{equation}
By the direct computations, we have
\[
\text{\rm div} (T)(x) =\frac12 \frac1{|x|^4 E_1^2\lt(\frac{|x|}R\rt)} \lt(1 + E_2^{-1}\lt(\frac{|x|}R\rt) + E_2^{-2}\lt(\frac{|x|}R\rt)\rt).
\]
Using the simple inequality $(1 + t)^\alpha \leq 1 + \alpha t + \frac{\alpha(\alpha-1)}2 t^2$ for $t \geq 0$ and $\alpha \in (1,2]$, we get
\[
||x|T(x)|^{2}=\frac14 \frac1{|x|^4 E_1^2\lt(\frac{|x|}R\rt)}\lt( 1+ 2 E_2^{-1}\lt(\frac{|x|}R\rt) +  E_2^{-2}\lt(\frac{|x|}R\rt) \rt).
\]
So, we have
\begin{align*}
\text{\rm div} (T)(x) - ||x|T(x)|^{2} &= \frac14 \frac1{|x|^4 E_1^2\lt(\frac{|x|}R\rt)}\lt(1 +  E_2^{-2}\lt(\frac{|x|}R\rt)\rt).
\end{align*}
Consequently, it holds
\begin{align}\label{eq:div1a}
\frac14 \int_{B} \frac{|u_k|^2}{|x|^4 E_1^2\lt(\frac{|x|}R\rt)} dx &+ \frac14 \int_B\frac{|u_k|^2}{|x|^4 E_1^2\lt(\frac{|x|}R\rt)E_2^{2}\lt(\frac{|x|}R\rt)} dx \notag\\
&\qquad\qquad\leq \int_\Om |u_k|^2 \text{\rm div} (T)(x) dx - \int_\Om |u_k|^2 ||x|T(x)|^{2} dx.
\end{align}
By integration by parts, we have
\[
\int_B |u_k|^2 \text{\rm div} (T)(x) dx = -2\int_B  u_k \la \na u_k, T\ra dx.
\]
Using H\"older inequality, we obtain
\begin{align}\label{eq:Holder1a}
\lt|\int_B |u_k|^2 \text{\rm div} (T)(x) dx\rt|& \leq 2\int_B |u_k| |x| |T(x)|\frac{|\na u(x)|}{|x|} dx\notag\\
&\leq  \int_B \frac{|\na u_k(x)|^2}{|x|^{2}} dx + \int_B |u_k|^2 ||x|T(x)|^{2} dx.
\end{align}
Combining \eqref{eq:Rellichgradient}, \eqref{eq:div1a} and \eqref{eq:Holder1a} together, we obtain \eqref{eq:dkdu}. This finishes the proof of \eqref{eq:criticalRellich}.

We next prove the second statement. Given $\alpha_1, \alpha_2 >0$, we define
\[
w(x) = E_1^{\frac{1-\al_1}2}\lt(\frac{|x|}R\rt) E_2^{\frac{1-\alpha_2}2}\lt(\frac{|x|}R\rt).
\]
Suppose that $B_\de \subset \Om$ for some $\de >0$. Let $\vphi$ be a cut-off function in $B$, i.e., $\vphi \in C_0^\infty(B)$ is radial function, $0 \leq \vphi\leq 1$ and $\vphi \equiv 1$ in $B_{1/2}$. For $a >0$ denote $\vphi_a(x) = \vphi(x/a)$. We define $u(x) = \vphi_\de(x) w(x)$, $u_\alpha(x) = |x|^{\alpha} u(x)$, $\alpha >0$ and $u_{\al,\ep}(x) = u_\alpha(x) (1-\vphi_\ep(x))$ for $\ep < \frac \de 2$. Notice that $u_{\al,\ep} \in C_0^\infty(\Om\setminus\{0\})$ and its support is contained in $B_\de \setminus B_\epsilon$.

We have
\begin{align*}
\na E_1\lt(\frac{|x|}R\rt) = -\frac{x}{|x|^2}, \quad \De E_1\lt(\frac{|x|}R\rt)  = -\frac2{|x|^2}, \quad \na E_2\lt(\frac{|x|}R\rt) = -\frac{1}{E_1\lt(\frac{|x|}R\rt)} \frac x{|x|^2}
\end{align*}
and
\[
\De E_2\lt(\frac{|x|}R\rt) = -\frac1{E_1^2\lt(\frac{|x|}R\rt) |x|^2} - \frac2{E_1\lt(\frac{|x|}R\rt) |x|^2}.
\]
Hence, for $x \not=0$ we have
\begin{align}\label{eq:Dew}
\De w(x) &= -(1-\al_1)\lt(1 + \frac{1+\al_1}4 E_1^{-1}\lt(\frac{|x|}R\rt)\rt)E_1^{-\frac{1+\al_1}2}\lt(\frac{|x|}R\rt)E_2^{\frac{1-\al_2}2}\lt(\frac{|x|}R\rt) |x|^{-2}\notag\\
&\quad -(1-\al_2) \lt(1 + \frac1{2E_1\lt(\frac{|x|}R\rt)} +\frac{1+ \al_2}4 \frac1{E_1\lt(\frac{|x|}R\rt)E_2\lt(\frac{|x|}R\rt) }\rt)\notag\\
&\qquad\qquad\qquad\times E_1^{-\frac{1+\al_1}2}\lt(\frac{|x|}R\rt)E_2^{-\frac{1+\al_2}2}\lt(\frac{|x|}R\rt) |x|^{-2}\notag\\
&\quad -\frac{(1-\al_1)(1-\al_2)}4 E_1^{-1-\al_1}\lt(\frac{|x|}R\rt)E_2^{-\al_2}\lt(\frac{|x|}R\rt) |x|^{-2}.
\end{align}
Moreover, using the polar coordinate and making the change of variable $t = \ln E_1(r/R)$, we have
\begin{align*}
\int_B \lt(E_1^{-\frac{1+\al_1}2}\lt(\frac{|x|}R\rt)E_2^{-\frac{1-\al_2}2}\lt(\frac{|x|}R\rt) |x|^{-2}\rt)^2 dx &= 2\pi^2 \int_0^1 E_1^{-1-\al_1}(r/R) E_2^{-1+\al_2}(r/R) r^{-1} dr\\
&= 2\pi^2 \int_{\ln(E_1(1/R))}^\infty e^{-\al_1 t} (1+t)^{-1+ \al_2} dt\\
&< \infty
\end{align*}
since $\al_1 >0$. Similarly, it holds
\[
\int_B \lt(E_1^{-1-\al_1}\lt(\frac{|x|}R\rt)E_2^{-\al_2}\lt(\frac{|x|}R\rt) |x|^{-2}\rt)^2 dx < \infty.
\]
Using $E_1, E_2 \geq 1$, it is easy to check that $\int_{B\setminus\{0\}} (\De w)^2 dx < \infty$. Consequently, we have
\begin{equation}\label{eq:khatichDeu}
\int_{\Om\setminus\{0\}} (\De u)^2 dx < \infty.
\end{equation}

Note that $u_\al -u = (|x|^\al-1) u$ and hence
\[
\De(u_\al -u) = (|x|^\al -1)\De u + 2\al |x|^{\al-2} \la x, \na u\ra + \al(n+ \al-2) |x|^{\al-2} u.
\]
By \eqref{eq:khatichDeu} and the Lebesgue's dominated convergence theorem, we have
\begin{equation}\label{eq:alpha01}
\lim_{\alpha \to 0} \int_{\Om\setminus\{0\}} (|x|^\al -1)^2 (\De u)^2 dx =0.
\end{equation}
Since $0\leq \varphi \leq 1$, $E_2\geq 1$ and $E_2 \leq C_{\al_1} E_1^{\frac {\al_1}2}$ for some constant $C_{\al_1} >0$, then it holds
\[
\int_{\Om\setminus\{0\}} |x|^{2\al -4} u^2 dx \leq C_{\al_1}\int_{B\setminus\{0\}} |x|^{2\alpha -4} E_1^{1-\frac{\al_1}2}\lt(\frac{|x|}R\rt)dx.
\]
Using polar coordinate and integration by parts, we get
\begin{align*}
\int_{\Om\setminus\{0\}} |x|^{2\al -4} u^2 dx &\leq C_{\al_1}2\pi^2 \int_0^1 r^{2\al -1} E_1^{1 -\frac{\al_1}2}(r/R) dr\\
&= \frac{C_{\al_1}}{2\al} E_1^{1-\frac{\al_1}2}(1/R) + \frac{(1-\frac{\al_1}2)C_{\al_1}}{2\al} \int_0^1 r^{2\al -1} E_1^{-\frac{\al_1}2}(r/R) dr.
\end{align*}
Since $E_1 \geq 1$ and $E_1$ is decreasing on $(0,1)$, for any $a \in (0,1)$, we have
\begin{align*}
\int_0^1 r^{2\al -1} E_1^{-\frac{\al_1}2}(r/R) dr &\leq \int_a^1 r^{2\al -1} dr + \int_0^{a} r^{2\al-1} dr E_1^{-\frac{\al_1}2}(a/R)\\
&= \frac{1-a^{2\al}}{2\al} + \frac{a^{2\al}}{2\al} E_1^{-\frac{\al_1}2}(a/R).
\end{align*}
So, we now have
\[
\int_{\Om\setminus\{0\}} |x|^{2\al -4} u^2 dx \leq \frac{C_{\al_1}}{2\al} E_1^{1-\frac{\al_1}2}(1/R) + \frac{(1-\frac{\al_1}2)C_{\al_1}(1-a^{2\al})}{4\al^2} + \frac{(1-\frac{\al_1}2)C_{\al_1}a^{2\al}}{4\al^2} E_1^{-\frac{\al_1}2}(a/R),
\]
for any $a \in (0,1)$. Multiplying both sides by $\al^2$ and then letting $\al \to 0$, we get
\[
\limsup_{\al \to 0} \alpha^2 \int_{\Om\setminus\{0\}} |x|^{2\al -4} u^2 dx \leq \frac{(1-\frac{\al_1}2)C_{\al_1}a^{2\al}}{4} E_1^{-\frac{\al_1}2}(a/R).
\]
Letting $a \to 0$, we obtain
\begin{equation}\label{eq:alpha02}
\lim_{\al \to 0} (\alpha(n+ \al -2))^2 \int_{\Om\setminus\{0\}} |x|^{2\al -4} u^2 dx = 0.
\end{equation}
Since 
\[
\na u = w \na \varphi_\de -\vphi_\de E_1^{-\frac{1+ \al_1}2}\lt(\frac{|x|}R \rt)E_2^{\frac{1-\al_2}2}\lt(\frac{|x|}R \rt) \frac x{|x|^2} \lt(\frac{1-\al_1}2 + \frac{1-\al_2}{2E_2\lt(\frac{|x|}R\rt)}\rt),
\]
and the support of $\na \varphi_\de$ is contained in $B_\de \setminus B_{\de/2}$, hence it is easy to prove that
\begin{equation}\label{eq:alpha000}
\lim_{\al \to 0} \alpha^2 \int_{\Om\setminus\{0\}} \lt(|x|^{\al -2} \la x , \na \vphi_\de\ra w\rt)^2 dx = 0.
\end{equation}
In other hand
\begin{align*}
&\int_{\Om\setminus\{0\}} \lt(|x|^{\al -2} \la x , \vphi_\de E_1^{-\frac{1+ \al_1}2}\lt(\frac{|x|}R \rt)E_2^{\frac{1-\al_2}2}\lt(\frac{|x|}R \rt) \frac x{|x|^2} \lt(\frac{1-\al_1}2 + \frac{1-\al_2}{2E_2\lt(\frac{|x|}R\rt)}\rt)\ra \rt)^2 dx\\
&\leq \lt(\frac{2-\al_1 -\al_2}2\rt)^2 C_{\al_1} \int_{B\setminus\{0\}} |x|^{-4} E_1^{-1 -\frac{\al_1}2}\lt(\frac{|x|}R\rt) dx\\
&=\lt(\frac{2-\al_1 -\al_2}2\rt)^2 C_{\al_1} 2\pi^2 \int_0^1 r^{-1} E_1^{-1-\frac{\al_1}2}(r/R) dr\\
&=\lt(\frac{2-\al_1 -\al_2}2\rt)^2 C_{\al_1} 2\pi^2 \int_{\ln E_1(1/R)}^\infty e^{-\frac{\al_1}2 t} dt,
\end{align*}
here we use $E_2 \geq 1$, $\al_1, \al_2 > 0$ small enough and the change of variable $t = \ln E_1(r/R)$. Consequently, we have
\begin{align}\label{eq:alpha001}
\lim_{\al\to 0} \al^2 &\int_{B\setminus\{0\}} \lt(|x|^{\al -2} \la x , \vphi_\de E_1^{-\frac{1+ \al_1}2}\lt(\frac{|x|}R \rt)E_2^{\frac{1-\al_2}2}\lt(\frac{|x|}R \rt) \frac x{|x|^2} \lt(\frac{1-\al_1}2 + \frac{1-\al_2}{2E_2\lt(\frac{|x|}R\rt)}\rt)\ra \rt)^2 dx \notag\\
&= 0.
\end{align}
Combining \eqref{eq:alpha000} and \eqref{eq:alpha001} together, we get
\begin{equation}\label{eq:alpha03}
\lim_{\al\to 0} \int_{\Om\setminus\{0\}} (2 \alpha |x|^{\al -2} \la x, \na u\ra)^2 dx = 0.
\end{equation}
It follows from \eqref{eq:alpha01}, \eqref{eq:alpha02} and \eqref{eq:alpha03} that $\int_{B\setminus\{0\}} (\De(u_\al -u))^2 dx \to 0$ as $\alpha \to 0$. In particular, we get
\begin{equation}\label{eq:alphato01}
\lim_{\alpha \to 0} \int_{\Om\setminus\{0\}} (\De u_\al)^2 dx =  \int_{B\setminus\{0\}} (\De u)^2 dx.
\end{equation}
Since 
\begin{align*}
\int_\Om \frac{u^2}{|x|^4 E_1^2(|x|/R)} dx & = \int_{B_\de} |x|^{-4} \varphi_\de^2 E_1^{-1 -\al_1}\lt(\frac{|x|}R\rt) E_2^{1-\al_2}\lt(\frac{|x|}R\rt) dx\\
&\leq C_{\al_1} 2\pi^2 \int_0^\de r^{-1} E_1^{-1-\frac{\al_1}2}(r/R) dr\\
&= C_{\al_1} 2\pi^2 \int_{\ln E_1(\de/R)} e^{-\frac{\al_1}2 t} dt\\
&< \infty,
\end{align*}
then by Lebesgue's dominated convergence theorem, we have
\begin{equation}\label{eq:alphato02}
\lim_{\alpha \to 0} \int_\Om \frac{u_\al^2}{|x|^4 E_1^2(|x|/R)} dx =  \int_\Om \frac{u^2}{|x|^4 E_1^2(|x|/R)} dx,
\end{equation}
and
\begin{equation}\label{eq:alphato03}
\lim_{\alpha \to 0} \int_\Om \frac{u_\al^2}{|x|^4 E_1^2(|x|/R)E_2^2(|x|/R)} dx =  \int_\Om \frac{u^2}{|x|^4 E_1^2(|x|/R)E_2^2(|x|/R)} dx.
\end{equation}

Note that $u_{\al,\ep} \in C_0^\infty(\Om \setminus\{0\})$ and 
\[
\De(u_{\al,\ep} -u_\al) = -\varphi_\ep \De u_\al - 2 \la \na \vphi_\ep, \na u_\al\ra + u_\al \De \vphi_\ep.
\]
The Lebesgue's dominated convergence theorem implies 
\begin{equation}\label{eq:epto01}
\lim_{\ep \to 0} \int_\Om \varphi_\ep^2 (\De u_\al)^2 dx =0.
\end{equation}
Since $\De \varphi_\ep(x) = \ep^{-2} (\De \varphi)(x/\ep)$, we then have
\begin{align*}
\int_\Om u_\al^2 (\De \vphi_\ep)^2 dx &\leq (\sup \De \vphi)^2 \ep^{-4} \int_{B_\ep \setminus B_{\ep/2}} u_\al^2 dx \\
&\leq (\sup \De \vphi)^2\frac{ 2\pi^2}{4+ 2\al} \ep^{2\alpha} E_1^{1-\al_1}\lt(\frac {\ep}{2R}\rt) E_2^{1-\al_2}\lt(\frac{\ep}{2R}\rt).
\end{align*}
Therefore, it holds
\begin{equation}\label{eq:epto02}
\lim_{\ep \to 0} \int_\Om u_\al^2 (\De \vphi_\ep)^2 dx =0.
\end{equation}
By direct computations, we get
\begin{align*}
\na u_\al &= \varphi_\de |x|^{\al-1} E_1^{\frac{1-\al_1}2}\lt(\frac{|x|}R\rt)E_2^{\frac{1-\al_2}2}\lt(\frac{|x|}R\rt) \frac{x}{|x|}\lt(\al - \frac{1-\al_1}{2 E_1\lt(\frac{|x|}R\rt)} - \frac{1-\al_2}{2 E_1\lt(\frac{|x|}R\rt)E_2\lt(\frac{|x|}R\rt)}\rt)\\
&\quad + |x|^\al w \na \vphi_\de.
\end{align*}
For $\ep >0$ small enough, we have $\la \na \vphi_\ep, \na \varphi_\de\ra =0$. Since $E_1, E_2 \geq 1$, then 
\begin{align*}
|\la \na \vphi_\ep, \na u_\al\ra|&\leq \frac{(2 + 2\al -\al_1 -\al_2)\sup |\na \vphi|}{2\ep} |x|^{\al -1} E_1^{\frac{1-\al_1}2}\lt(\frac{\ep}{2R}\rt)E_2^{\frac{1-\al_2}2}\lt(\frac{\ep}{2R}\rt) \chi_{B_\ep \setminus B_{\ep/2}}(x).
\end{align*}
Thus, we have
\begin{align*}
\int_\Om &|\la \na \vphi_\ep, \na u_\al\ra|^2 dx \\
&\qquad\leq \frac{((2 + 2\al -\al_1 -\al_2)\sup |\na \vphi|)^2}{4\ep^2} E_1^{1-\al_1}\lt(\frac{\ep}{2R}\rt)E_2^{1-\al_2}\lt(\frac{\ep}{2R}\rt)\int_{B_\ep} |x|^{2 \al -2} dx\\
&\qquad= \frac{((2 + 2\al -\al_1 -\al_2)\sup |\na \vphi|)^2}{4(2 + 2\al)} \ep^{2\al} E_1^{1-\al_1}\lt(\frac{\ep}{2R}\rt)E_2^{1-\al_2}\lt(\frac{\ep}{2R}\rt)
\end{align*}
which then implies  
\begin{equation}\label{eq:epto03}
\lim_{\ep \to 0} \int_\Om |\la \na \vphi_\ep, \na u_\al\ra|^2 dx =0.
\end{equation}
Combining \eqref{eq:epto01}, \eqref{eq:epto02} and \eqref{eq:epto03} together yields $\int_\Om (\De(u_{\al,\ep} -u_\al))^2 dx \to 0$ as $\ep \to 0$. In particular, we have
\begin{equation}\label{eq:epto001}
\lim_{\ep \to 0} \int_\Om (\De u_{\al,\ep})^2 dx = \int_{\Om\setminus\{0\}} (\De u_\al)^2 dx.
\end{equation}
By the Lebesgue's dominated convergence theorem, we have
\begin{equation}\label{eq:epto002}
\lim_{\ep \to 0} \int_\Om \frac{u_{\al,\ep}^2}{|x|^4 E_1^2(|x|/R)} dx =  \int_\Om \frac{u_\al^2}{|x|^4 E_1^2(|x|/R)} dx,
\end{equation}
and
\begin{equation}\label{eq:epto003}
\lim_{\ep \to 0} \int_\Om \frac{u_{\al,\ep}^2}{|x|^4 E_1^2(|x|/R)E_2^2(|x|/R)} dx =  \int_\Om \frac{u_\al^2}{|x|^4 E_1^2(|x|/R)E_2^2(|x|/R)} dx.
\end{equation}

Now, suppose that the inequality \eqref{eq:bestexponent} holds in $C_0^\infty(\Om \setminus\{0\})$ for some $D >0$ and $\gamma \in \R$. Applying the inequality \eqref{eq:bestexponent} for the function $u_{\al, \ep}$, and letting $\ep \to 0$ and then $\alpha \to 0$ and using the limits \eqref{eq:epto001}, \eqref{eq:epto002}, \eqref{eq:epto003}, \eqref{eq:alphato01}, \eqref{eq:alphato02} and \eqref{eq:alphato03}, we obtain
\begin{equation}\label{eq:cuthe}
\int_{\Om\setminus\{0\}} (\De u)^2 dx - \int_\Om \frac{|u|^2}{|x|^4 E_1^2\lt(\frac{|x|}R\rt)} dx \geq B\int_\Om \frac{|u|^2}{|x|^4 E_1^2\lt(\frac{|x|}R\rt) E_2^\gamma \lt(\frac{|x|}R\rt)} dx,
\end{equation}
with $u(x) = \varphi_\de(x) w(x)$. Since 
\[
\De u = \varphi_\de \De w + 2 \la \na \varphi_\de, \na w\ra + w \De \varphi_\de,
\]
and the supports of $\na \varphi_\de$ and $\De \varphi_\de$ are contained in $B_\de \setminus B_{\de/2}$, then we have
\begin{equation}\label{eq:tiemcan1}
\int_{\Om\setminus\{0\}} (\De u)^2 dx = \int_{\Om\setminus\{0\}} \varphi_\de^2 (\De w)^2 dx + O(1),
\end{equation}
where $O(1)$ denotes the quantity which is uniformly bounded when $\al_1, \al_2 \to 0$. Furthermore, from \eqref{eq:Dew} we have
\begin{align*}
(\De w)^2 &\leq (1-\al_1)^2 E_1^{-1-\al_1}\lt(\frac{|x|}R\rt)E_2^{1-\al_2}\lt(\frac{|x|}R\rt) |x|^{-4} \\
&\quad+ (1-\al_2)^2 E_1^{-1-\al_1}\lt(\frac{|x|}R\rt)E_2^{-1-\al_2}\lt(\frac{|x|}R\rt) |x|^{-4}\\
&\quad + 2(1-\al_1)(1-\al_2) E_1^{-1-\al_1}\lt(\frac{|x|}R\rt)E_2^{-\al_2}\lt(\frac{|x|}R\rt) |x|^{-4} \\
&\quad + C E_1^{-\frac{3}2}\lt(\frac{|x|}R\rt) E_2\lt(\frac{|x|}R\rt)|x|^{-4},
\end{align*}
with the positive constant $C > 0$ independent of $\al_1$ and $\al_2$ (when they tend to $0$). Integrating both sides on $\Om \setminus\{0\}$ and using the definition of $u$, we get
\begin{align}\label{eq:tiemcan2}
\int_{\Om\setminus\{0\}} \vphi_\de^2 (\De w)^2 &\leq (1 -\al_1)^2 \int_{\Om} \frac{u^2}{|x|^4 E_1^2\lt(\frac{|x|}R\rt)} dx + (1-\al_2)^2 \int_{\Om} \frac{u^2}{|x|^4 E_1^2\lt(\frac{|x|}R\rt)E_2^2\lt(\frac{|x|}R\rt)} dx\notag\\
&\quad + 2(1-\al_1)(1-\al_2) \int_{\Om} \varphi_\de^2 E_1^{-1-\al_1}\lt(\frac{|x|}R\rt)E_2^{-\al_2}\lt(\frac{|x|}R\rt) |x|^{-4} dx + O(1).
\end{align}
Using the polar coordinate, the change of variable $t = \ln E_1(r/R)$ and integration by parts, we have
\begin{align*}
\int_{\Om} \frac{u^2}{|x|^4 E_1^2\lt(\frac{|x|}R\rt)} dx&= \int_{\Om} \varphi_\de^2 E_1^{-1-\al_1}(|x|/R) E_2^{1-\al_2}(|x|/R) |x|^{-4} dx\\
&= 2\pi^2 \int_0^\de \varphi_\de(r)^2 E_1^{-1-\al_1}(r/R) E_2^{1-\al_2}(r/R) r^{-1} dr\\
&= 2\pi^2 \int_{\ln E_1(\de/R)}^\infty \varphi_\de (R e^{-(e^t-1)}) e^{-\al_1 t} (1+t)^{1-\al_2} dt\\
&= \frac{1-\al_2}{\al_1} 2\pi^2 \int_{\ln E_1(\de/R)}^\infty \varphi_\de (R e^{-(e^t-1)}) e^{-\al_1 t} (1+t)^{-\al_2} dt + O(1)\\
&=\frac{1-\al_2}{\al_1}\int_{\Om} E_1^{-1-\al_1}\lt(\frac{|x|}R\rt)E_2^{-\al_2}\lt(\frac{|x|}R\rt) |x|^{-4} dx + O(1).
\end{align*}
Inserting this estimate into \eqref{eq:tiemcan2} yields
\begin{align}\label{eq:tiemcan3}
\int_{\Om\setminus\{0\}} \vphi_\de^2 (\De w)^2 &\leq \int_{\Om} \frac{u^2}{|x|^4 E_1^2\lt(\frac{|x|}R\rt)} dx + (1-\al_2)^2 \int_{\Om} \frac{u^2}{|x|^4 E_1^2\lt(\frac{|x|}R\rt)E_2^2\lt(\frac{|x|}R\rt)} dx\notag\\
&\quad -\al_1(1-\al_2) \int_{\Om} \varphi_\de^2 E_1^{-1-\al_1}\lt(\frac{|x|}R\rt)E_2^{-\al_2}\lt(\frac{|x|}R\rt) |x|^{-4} dx + O(1)\notag\\
&\leq  \int_{\Om} \frac{u^2}{|x|^4 E_1^2\lt(\frac{|x|}R\rt)} dx + (1-\al_2)^2 \int_{\Om} \frac{u^2}{|x|^4 E_1^2\lt(\frac{|x|}R\rt)E_2^2\lt(\frac{|x|}R\rt)} dx + O(1).
\end{align}
Inserting \eqref{eq:tiemcan1} and \eqref{eq:tiemcan3} into \eqref{eq:cuthe} we obtain
\begin{align}\label{eq:tobest}
D\int_\Om \frac{|u|^2}{|x|^4 E_1^2\lt(\frac{|x|}R\rt) E_2^\gamma \lt(\frac{|x|}R\rt)} dx \leq (1-\al_2)^2 \int_{\Om} \frac{u^2}{|x|^4 E_1^2\lt(\frac{|x|}R\rt)E_2^2\lt(\frac{|x|}R\rt)} dx + O(1).
\end{align}
Suppose that $\gamma < 2$. We have
\begin{align*}
\int_\Om \frac{|u|^2}{|x|^4 E_1^2\lt(\frac{|x|}R\rt) E_2^\gamma \lt(\frac{|x|}R\rt)} dx&\geq 2\pi^2 \int_0^{\de/2} r^{-1} E_1^{-1-\al_1}(r/R) E_2^{1-\al_2 -\gamma}(r/R) dr\\
&= 2\pi^2 \int_{\ln E_1(\de/(2R))} ^\infty e^{-\al_1 t} (1+ t)^{1-\al_2 -\gamma} dt
\end{align*}
here we use $\varphi_\de =1$ in $B_{\de/2}$ and the change of variable $t = \ln E_1(r/R)$. Taking $0 < \al_2 < 2 -\gamma$ small enough, we then have
\[
\lim_{\al_1 \to 0} \int_\Om \frac{|u|^2}{|x|^4 E_1^2\lt(\frac{|x|}R\rt) E_2^\gamma \lt(\frac{|x|}R\rt)} dx = \infty,
\]
while
\begin{align*}
\int_\Om \frac{|u|^2}{|x|^4 E_1^2\lt(\frac{|x|}R\rt) E_2^2 \lt(\frac{|x|}R\rt)} dx&\leq 2\pi^2 \int_0^{\de} r^{-1} E_1^{-1-\al_1}(r/R) E_2^{-1-\al_2}(r/R) dr\\
&= 2\pi^2 \int_{\ln E_1(\de/R)} ^\infty e^{-\al_1 t} (1+ t)^{-1-\al_2} dt
\end{align*}
which is bounded when $\al_1 \to 0$ since $\al_2 > 0$. This contradicts to \eqref{eq:tobest}. Thus, we get $\gamma \geq 2$.

If $\gamma =2$, then dividing both sides of \eqref{eq:tobest} by $\int_\Om \frac{|u|^2}{|x|^4 E_1^2\lt(\frac{|x|}R\rt) E_2^2 \lt(\frac{|x|}R\rt)} dx$, we get
\begin{equation}\label{eq:tobest1}
D \leq (1-\al_2)^2 + \frac{O(1)}{\int_\Om \frac{|u|^2}{|x|^4 E_1^2\lt(\frac{|x|}R\rt) E_2^2 \lt(\frac{|x|}R\rt)} dx}.
\end{equation}
Notice that 
\begin{align*}
\int_\Om \frac{|u|^2}{|x|^4 E_1^2\lt(\frac{|x|}R\rt) E_2^2 \lt(\frac{|x|}R\rt)} dx&\geq 2\pi^2 \int_0^{\de/2} r^{-1} E_1^{-1-\al_1}(r/R) E_2^{-1-\al_2}(r/R) dr\\
&= 2\pi^2 \int_{\ln E_1(\de/(2R))} ^\infty e^{-\al_1 t} (1+ t)^{-1-\al_2} dt.
\end{align*}
It is easy to see that 
\[
\lim_{\al_2\to 0} \lim_{\al_1\to 0} \int_\Om \frac{|u|^2}{|x|^4 E_1^2\lt(\frac{|x|}R\rt) E_2^2 \lt(\frac{|x|}R\rt)} dx = \infty.
\]
Letting $\al_1 \to 0$ and then $\al_2 \to 0$, we obtain $D \leq 1$. 

The proof is completed.
\end{proof}

\section{Estimate of the $L^q$ norm}
The following proposition is the main result of this section. It provides an estimate for $L^q$ norm of $E_2^{-1}(|x|/R) u(x)$ which helps us to prove the first part of our main theorem.
\begin{proposition}\label{Lqnorm}
Let $u\in C_0^\infty(\Om\setminus\{0\})$, then for any $R \geq R_\Om$ and $q >2$ we have the following estimate,
\begin{equation}\label{eq:Lqnorm}
\lt(\int_\Om \bigg|\frac{u(x)}{E_2\lt(\frac{|x|}R\rt)}\bigg|^q dx\rt)^{\frac1q} \leq \frac{1}{4\sqrt{2}\pi} \lt(\sqrt{\frac32} + \frac94 + \sqrt{\frac13}\rt) \lt(1 + \frac q2\rt)^{\frac12 + \frac1q} |\Om|^{\frac1q} (\tilde I_4[u,\Om,R])^{\frac12}.
\end{equation}
\end{proposition}

\begin{proof}
We follow the argument in \cite{MT2018,PS2015}. Define $v(x) = E_1^{-\frac12}\lt(\frac{|x|}R\rt) u(x)$. Since $u\in C_0^\infty(\Om\setminus\{0\})$, then so is $\frac{u(x)}{E_2\lt(\frac{|x|}R\rt)}$. By the formula \eqref{eq:expression2} and the definition of $v$, we have
\begin{align*}
\frac{u(x)}{E_2\lt(\frac{|x|}R\rt)} &= \frac1{4 \pi^2} \int_\Om\frac{\De \lt(\frac{v(y) E_1^{\frac12}\lt(\frac{|y|}R\rt)}{E_2\lt(\frac{|y|}R\rt)}\rt) }{|x-y|^2} dy\\
&= \frac1{4\pi^2} \int_\Om \frac{\De v(y) \frac{ E_1^{\frac12}\lt(\frac{|y|}R\rt)}{E_2\lt(\frac{|y|}R\rt)}}{|x-y|^2} dy + \frac1{4 \pi^2} \int_\Om \frac{\De \lt(\frac{E_1^{\frac12}\lt(\frac{|y|}R\rt)}{E_2\lt(\frac{|y|}R\rt)}\rt) }{|x-y|^2} v(y) dy\\
&\quad + \frac1{2\pi^2} \int_{\Om} \frac{\la \na v,\na \lt(\frac{E_1^{\frac12}\lt(\frac{|y|}R\rt)}{E_2\lt(\frac{|y|}R\rt)}\rt)\ra}{|x-y|^2} dy. 
\end{align*}
Since $E_2 \geq 1$, we then have
\begin{align}\label{eq:estu1}
\lt|\frac{u(x)}{E_2\lt(\frac{|x|}R\rt)}\rt|\leq &\frac1{4\pi^2} \int_\Om \frac{|\De v(y)| E_1^{\frac12}\lt(\frac{|y|}R\rt)}{|x-y|^{2}} dy +  \frac1{4 \pi^2} \int_\Om \frac{\lt|\De \lt(\frac{E_1^{\frac12}\lt(\frac{|y|}R\rt)}{E_2\lt(\frac{|y|}R\rt)}\rt)\rt| }{|x-y|^2} |v(y)| dy\notag\\
& + \frac1{2 \pi^2} \int_\Om \frac{|\na v(y)|}{|x-y|^{2}} \lt|\na \lt(\frac{E_1^{\frac12}\lt(\frac{|y|}R\rt)}{E_2\lt(\frac{|y|}R\rt)}\rt)\rt| dy.
\end{align}
Using \eqref{eq:daoham}, we can readily check that
\[
\lt(\frac{E_1^{\frac12}}{E_2}\rt)'(t) = -\frac{E_2(t) -2}{2 t E_1^{\frac12}(t) E_2^{2}(t)}.
\] 
Therefore, it holds
\[
\na \lt(\frac{E_1^{\frac12}\lt(\frac{|x|}R\rt)}{E_2\lt(\frac{|y|}R\rt)}\rt) = -\frac{y}{|y|^2}\frac{E_2\lt(\frac{|y|}R\rt) -2}{2 E_1^{\frac12}\lt(\frac{|y|}R\rt) E_2^{2}\lt(\frac{|y|}R\rt)}.
\]
Using again $E_2 \geq 1$ and the simple inequality 
\begin{equation}\label{eq:ele}
|t-2| \leq t,\quad t\geq 1,
\end{equation}
we obtain
\begin{equation}\label{eq:estE}
\lt|\na \lt(\frac{E_1^{\frac12}\lt(\frac{|x|}R\rt)}{E_2\lt(\frac{|y|}R\rt)}\rt)\rt| \leq \frac{1}{2|y| E_1^{\frac12}\lt(\frac{|y|}R\rt) E_2\lt(\frac{|y|}R\rt)}.
\end{equation}
Using again \eqref{eq:daoham} we have
\[
\lt(\frac{E_1^{\frac12}}{E_2}\rt)''(t) =\frac1{2t^2E_1^{\frac32}(t) E_2^2(t)} + \frac{E_2(t) -2}{2 t^2 E_1^{\frac12}(t) E_2^{2}(t)} -\frac{E_2(t) -2}{4 t^2 E_1^{\frac32}(t) E_2^{2}(t)} -\frac{E_2(t) -2}{ t^2 E_1^{\frac32}(t) E_3^{3}(t)}.
\] 
Therefore, it holds
\begin{align*}
\De \lt(\frac{E_1^{\frac12}\lt(\frac{|y|}R\rt)}{E_2\lt(\frac{|y|}R\rt)}\rt)&=-\frac{\lt(E_2\lt(\frac{|y|}R\rt) -2 + \frac{E_2\lt(\frac{|y|}R\rt)}{4E_1\lt(\frac{|y|}R\rt)}+\frac{E_2\lt(\frac{|y|}R\rt) -2}{E_1\lt(\frac{|y|}R\rt)E_2\lt(\frac{|y|}R\rt)}\rt)}{|y|^2 E_1^{\frac12}\lt(\frac{|y|}R\rt) E_2^{2}\lt(\frac{|y|}R\rt)}.
\end{align*}
Since $E_1, E_2 \geq 1$, then
\[
\lt|\De \lt(\frac{E_1^{\frac12}\lt(\frac{|y|}R\rt)}{E_2\lt(\frac{|y|}R\rt)}\rt)\rt|\leq \frac{1}{|y|^2 E_1^{\frac12}\lt(\frac{|y|}R\rt) E_2^{2}\lt(\frac{|y|}R\rt)}\lt(2\lt|E_2\lt(\frac{|y|}R\rt) -2\rt| + \frac{E_2\lt(\frac{|y|}R\rt)}{4}\rt).
\]
By \eqref{eq:ele}, we have
\begin{equation}\label{eq:estDeltaE}
\lt|\De \lt(\frac{E_1^{\frac12}\lt(\frac{|y|}R\rt)}{E_2\lt(\frac{|y|}R\rt)}\rt)\rt| \leq \frac{9}{4|y|^2 E_1^{\frac12}\lt(\frac{|y|}R\rt) E_2\lt(\frac{|y|}R\rt)}.
\end{equation}
Define
\[
K(x) = \int_\Om \frac{|\De v(y)| E_1^{\frac12}\lt(\frac{|y|}R\rt)}{|x-y|^{2}} dy,
\]
\[
L(x) = \int_\Om \frac{|\na v(y)| }{|y| E_1^{\frac12}\lt(\frac{|y|}R\rt) E_2\lt(\frac{|y|}R\rt) |x-y|^{2}} dy,
\]
and
\[
M(x) =\int_\Om\frac{|v(y)|}{|y|^2 E_1^{\frac12}\lt(\frac{|y|}R\rt) E_2\lt(\frac{|y|}R\rt) |x-y|^{2}} dy.
\]
Whence it follows from \eqref{eq:estu1}, \eqref{eq:estE} and \eqref{eq:estDeltaE} that 
\[
\lt|\frac{u(x)}{E_2\lt(\frac{|x|}R\rt)}\rt|\leq \frac1{4\pi^2} \lt(K(x) + L(x)+ \frac94 M(x)\rt),
\]
which implies by triangle inequality
\begin{equation}\label{eq:estu2}
\lt(\int_\Om \bigg|\frac{u(x)}{E_2\lt(\frac{|x|}R\rt)}\bigg|^q dx\rt)^{\frac1q} \leq \frac1{4\pi^2} \lt(\|K\|_{L^q(\Om)} + \|L\|_{L^q(\Om)}+ \frac94 \|M\|_{L^q(\Om)}\rt)
\end{equation}
Now, let $q >2$ and define $r$ by $\frac1n + \frac1r = \frac1q + 1$ or $r = 2q/(q+ 2)$. Then, clearly $1< r < 2$. For $x \in \Om$, let us define
\[
h_r(x) = \int_\Om |x-y|^{-2r} dy.
\]
Let $\tilde R$ be such that $|\Om| = |B_{\tilde R}|$. It was proved in \cite{MT2018} (see the proof of Proposition $3.1$) that  
\[
h_r(x) \leq \frac{2 \pi^2 \tilde R^{4 -2r}}{4-2r},\quad x \in \Om.
\]
Hence we have
\begin{equation}\label{eq:boundhr}
\|h_r\|_{L^\infty(\Om)}^{\frac1r} \leq \frac{\pi}{2^{\frac12}} \lt(1 + \frac{q}2\rt)^{\frac12 + \frac1q}  |\Om|^{\frac1q}.
\end{equation}
Let us break the integrand of $K$ as
\[
\frac{|\De v(y)| E_1^{\frac12}\lt(\frac{|x|}R\rt)}{|x-y|^{2}} =\lt(\frac{|\De v(y)|^2 E_1\lt(\frac{|x|}R\rt)}{|x-y|^{2r}}\rt)^{\frac1q}\lt(|\De v(y)|^2 E_1\lt(\frac{|x|}R\rt)\rt)^{\frac12 -\frac1q}|x-y|^{-2(1-\frac rq)}.
\]
Applying H\"older inequality with the exponents $q$, $2$ and $2q/(q-2)$ and note that $1-r/q = r/2$, we get
\begin{align*}
K(x) &\leq \lt(\int_\Om \frac{|\De v(y)|^2 E_1\lt(\frac{|y|}R\rt)}{|x-y|^{2r}} dy\rt)^{\frac1q} \lt(\int_\Om|\De v(y)|^2 E_1\lt(\frac{|x|}R\rt) dy\rt)^{\frac12 -\frac1q} \|h_r\|_{L^\infty(\Om)}^{\frac12}.
\end{align*}
Integrating $K(x)^q$ and using Fubini theorem, we obtain
\begin{align*}
\|K\|_{L^q(\Om)} &\leq \lt(\int_\Om \int_\Om \frac{|\De v(y)|^2 E_1\lt(\frac{|y|}R\rt)}{|x-y|^{2r}} dy dx\rt)^{\frac1q} \lt(\int_\Om|\De v(y)|^2 E_1\lt(\frac{|x|}R\rt) dy\rt)^{\frac12 -\frac1q} \|h_r\|_{L^\infty(\Om)}^{\frac12}\\
&= \lt(\int_\Om  |\De v(y)|^2 E_1\lt(\frac{|y|}R\rt) \lt(\int_\Om |x-y|^{-2r} dx\rt)dy\rt)^{\frac1q}\\
&\qquad\qquad\qquad\qquad\qquad \times  \lt(\int_\Om|\De v(y)|^2 E_1\lt(\frac{|x|}R\rt) dy\rt)^{\frac12 -\frac1q} \|h_r\|_{L^\infty(\Om)}^{\frac12}\\
&\leq \|h_r\|_{L^\infty(\Om)}^{\frac1r} \lt(\int_\Om|\De v(y)|^2 E_1\lt(\frac{|x|}R\rt) dy\rt)^{\frac12}.
\end{align*}
Now, we use Proposition \ref{GSrepresentation} to get
\begin{equation}\label{eq:boundK}
\|K\|_{L^q(\Om)} \leq \sqrt{\frac32} \|h_r\|_{L^\infty(\Om)}^{\frac1r} (\tilde I_4[u,\Om,R])^{\frac12}.
\end{equation}
Similarly, writing the integrand of $M(x)$ as
\begin{align*}
\frac{|v(y)|}{|y|^2 E_1^{\frac12}\lt(\frac{|y|}R\rt) E_2\lt(\frac{|y|}R\rt) |x-y|^{2}}& =\lt(\frac{|v(y)|^2}{|y|^4 E_1\lt(\frac{|y|}R\rt) E_2^2\lt(\frac{|y|}R\rt) |x-y|^{2r}}\rt)^{\frac1q} |x-y|^{-2(1-\frac rq)}\\
&\qquad \times \lt(\frac{|v(y)|^2}{|y|^4 E_1\lt(\frac{|y|}R\rt) E_2\lt(\frac{|y|}R\rt)}\rt)^{\frac12 -\frac1q},
\end{align*}
and applying H\"older inequality with the same exponents as in the case of $K(x)$ and noting that $v(x) = E_1^{-\frac12}(|x|/R) u(x)$, we get
\begin{align*}
M(x) &\leq \lt(\int_\Om\frac{|u(y)|^2}{|y|^4 E_1^2\lt(\frac{|y|}R\rt) E_2^2\lt(\frac{|y|}R\rt) |x-y|^{2r}} dy\rt)^{\frac1q} \|h_r\|_{L^\infty(\Om)}^{\frac12}\\
&\qquad \times \lt(\int_\Om \frac{|u(y)|^2}{|y|^4 E_1^2\lt(\frac{|y|}R\rt) E_2\lt(\frac{|y|}R\rt)}dy\rt)^{\frac12 -\frac1q}.
\end{align*}
Integrating $M(x)^q$, using Fubini theorem and \eqref{eq:criticalRellich}, we get
\begin{equation}\label{eq:boundM}
\|M\|_{L^q(\Om)} \leq \|h_r\|_{L^\infty(\Om)}^{\frac1r} (\tilde I_4[u,\Om,R])^{\frac12}.
\end{equation}
To conclude, we estimate $\|L\|_{L^q(\Om)}$. We estimate the integrand of $L(x)$ as 
\begin{align*}
\frac{|\na v(y)| }{|y| E_1^{\frac12}\lt(\frac{|y|}R\rt) E_2\lt(\frac{|y|}R\rt) |x-y|^{2}}&\leq \frac{E_1^{\frac12}\lt(\frac{|y|}R\rt)|\na v(y)| }{|y| |x-y|^{2}}\\
&=\lt(\frac{E_1\lt(\frac{|y|}R\rt)|\na v(y)|^2 }{|y|^2 |x-y|^{2r}}\rt)^{\frac1q} |x-y|^{-2(1-\frac rq)}\\
&\qquad \times \lt(\frac{E_1\lt(\frac{|y|}R\rt)|\na v(y)|^2 }{|y|^2}\rt)^{\frac12-\frac1q}.
\end{align*}
Applying H\"older inequality with the same exponents as in the case of $K(x)$, we get
\begin{align*}
L(x) \leq \lt(\int_\Om \frac{E_1\lt(\frac{|y|}R\rt)|\na v(y)|^2 }{|y|^2 |x-y|^{2r}} dy\rt)^{\frac1q}\|h_r\|_{L^\infty(\Om)}^{\frac12} \lt(\int_\Om\frac{E_1\lt(\frac{|y|}R\rt)|\na v(y)|^2 }{|y|^2} dy\rt)^{\frac12-\frac1q}.
\end{align*}
Integrating $L(x)^q$, using Fubini theorem and \eqref{eq:wcRellich}, we get
\begin{equation}\label{eq:boundL}
\|L\|_{L^q(\Om)} \leq \sqrt{\frac13} \|h_r\|_{L^\infty(\Om)}^{\frac1r} (\tilde I_4[u,\Om,R])^{\frac12}.
\end{equation}
Putting \eqref{eq:boundK}, \eqref{eq:boundM} and \eqref{eq:boundM} together with \eqref{eq:estu2} yields
\[
\lt(\int_\Om \bigg|\frac{u(x)}{E_2\lt(\frac{|x|}R\rt)}\bigg|^q dx\rt)^{\frac1q} \leq \frac{1}{4\sqrt{2}\pi} \lt(\sqrt{\frac32} + \frac94 + \sqrt{\frac13}\rt) \lt(1 + \frac q2\rt)^{\frac12 + \frac1q} |\Om|^{\frac1q} (\tilde I_4[u,\Om,R])^{\frac12}.
\]
This proves \eqref{eq:Lqnorm}.
\end{proof}

\section{Proof of Theorem \ref{Maintheorem}}
In this section, we prove Theorem \ref{Maintheorem}. The proof of the Leray--Adams inequality \eqref{eq:LerayAdams} follows the Trudinger's original proof of the Trudinger inequality by using the $L^q$ norm estimate from Proposition \ref{Lqnorm}. The second statement of Theorem \ref{Maintheorem} follows from Proposition \ref{criticalRellich}. Let us go to the detail of the proof.
\begin{proof}[Proof of Theorem \ref{Maintheorem}]
By density argument, it is enough to prove \eqref{eq:LerayAdams} for functions $u\in C_0^\infty(\Om\setminus\{0\})$ with $\tilde I_4[u,\Om,R] \leq 1$. Denote 
\[
C= \frac{1}{4\sqrt{2}\pi} \lt(\sqrt{\frac32} + \frac94 + \sqrt{\frac13}\rt).
\]
From Proposition \ref{Lqnorm}, we have
\[
\int_\Om \lt(\frac{|u|}{E_2(|x|/R)}\rt)^{2k} dx \leq C^{2k} (k+1)^{k+1} |\Om|,\quad k=2,3,\ldots.
\]
Multiplying both sides by $c^k/k!$ and adding from $2$ to $m$ with $m\geq 2$ we get
\[
\int_\Om \sum_{k=2}^m \frac1{k!} \lt(c\lt(\frac{|u|}{E_2(|x|/R)}\rt)^2\rt)^{k} dx \leq \lt(\sum_{k=2}^m (c C^2)^k \frac{(1+k)^k}{k!} \rt) |\Om|.
\]
Using Sterling formula, we have $k! \sim (k/e)^k \sqrt{2\pi k}$ as $k \to \infty$. Hence the right-hand side of the previous estimate converges if $c < (e C^2)^{-1}$. Hence, for $c < (e C^2)^{-1}$, by letting $m \to \infty$ we get
\[
\int_\Om \lt(e^{c\lt(\frac{|u|}{E_2(\frac{|x|}R)}\rt)^2} -1 -c\lt(\frac{|u|}{E_2(\frac{|x|}R)}\rt)^2\rt) dx\leq \lt(\sum_{k=2}^\infty (c C^2)^k \frac{(1+k)^k}{k!} \rt) |\Om|.
\]
Also, by H\"older inequality and Proposition \ref{Lqnorm} we have
\begin{align*}
\int_\Om \lt(\frac{|u|}{E_2(\frac{|x|}R)}\rt)^2 dx &\leq |\Om|^{\frac12} \lt(\int_\Om \lt(\frac{|u|}{E_2(\frac{|x|}R)}\rt)^4 dx\rt)^{\frac12} \leq C^2 3^{\frac 32} |\Om|.
\end{align*}
Adding these previous estimates, we get
\[
\int_\Om e^{c\lt(\frac{|u|}{E_2(\frac{|x|}R)}\rt)^2} dx \leq \lt(1+ c C^2 3^{\frac 32} + \sum_{k=2}^\infty (c C^2)^k \frac{(1+k)^k}{k!} \rt) |\Om|.
\]
This proves \eqref{eq:LerayAdams} for $\beta =1$. The case $\beta >1$ is followed immediately since $E_2 \geq 1$.

We next prove the second statement of Theorem \ref{Maintheorem}, i.e., that for $\beta < \frac12$, the inequality \eqref{eq:LerayAdams} is false. Suppose, for the sake of contradiction, that there exist $\beta < \frac12$ and two positive constants $c_1, c_2$ such that
\[
\int_\Om e^{c_1\lt(\frac{u}{E_1^\beta(\frac{|x|}R)}\rt)^2} dx \leq c_2 |\Om|,
\]
for any $u \in C_0^\infty(\Om)$ satisfying $\tilde I_4[u,\Om,R]\leq 1$. By scaling argument, we assume that $B \subset \Om\subset B_r$ for some $r >1$. We can choose $1 < \theta < 2$ such that $2\beta + \theta < 2$. Now, let $u \in C_0^\infty(B)$ be such that $\tilde I_4[u,\Om,R] \leq 1$. Then we have
\begin{align*}
\int_\Om \frac{u^2}{|x|^4 E_1^2(\frac{|x|}R) E_2^{2\beta + \theta}(\frac{|x|}R)} dx &= \frac1{c_1} \int_\Om \lt[c_1\frac{u^2}{E_2^{2\beta }(\frac{|x|}R)}\rt]\lt[\frac{1}{|x|^4 E_1^2(\frac{|x|}R) E_2^{\theta}(\frac{|x|}R)}\rt] dx\\
&\leq \frac1{c_1} \int_\Om e^{c_1\lt(\frac{u}{E_2^{\beta }(\frac{|x|}R)}\rt)^2 } dx + P_\beta,
\end{align*}
with
\[
P_\beta = \frac1{c_1} \int_\Om \lt(1+ \frac{1}{|x|^4 E_1^2(\frac{|x|}R) E_2^{\theta}(\frac{|x|}R)}\rt) \ln\lt(1 + \frac{1}{|x|^4 E_1^2(\frac{|x|}R) E_2^{\theta}(\frac{|x|}R)}\rt) dx,
\]
here we use the following version of Young's inequality
\[
ab \leq e^a -a -1  + (1+b) \ln(1+ b) -b,\quad a, b\geq 0.
\]
Notice that 
\[
P_\beta \leq \frac1{c_1} \int_{B_r} \lt(1+ \frac{1}{|x|^4 E_1^2(\frac{|x|}R) E_2^{\theta}(\frac{|x|}R)}\rt) \ln\lt(1 + \frac{1}{|x|^4 E_1^2(\frac{|x|}R) E_2^{\theta}(\frac{|x|}R)}\rt) dx.
\]
An easy calculation show that $P_\beta$ is bounded (see the proof of Theorem $1.1$ in \cite{PS2015}). Consequently, we have
\[
\int_\Om \frac{u^2}{|x|^4 E_1^2(\frac{|x|}R) E_2^{2\beta + \theta}(\frac{|x|}R)} dx \leq \frac{c_2}{c_1} |\Om| + P_\beta,
\]
for any $u \in C_0^\infty(\Om)$ satisfying $\tilde I_4[u,\Om,R]\leq 1$, which implies
\[
\frac1{\frac{c_2}{c_1} |\Om| + P_\beta} \int_\Om \frac{u^2}{|x|^4 E_1^2\lt(\frac{|x|}R\rt) E_2^{2\beta + \theta}\lt(\frac{|x|}R\rt)} dx \leq \int_\Om (\De u)^2 dx - \int_\Om \frac{|u|^2}{|x|^4 E_1^2\lt(\frac{|x|}R\rt)} dx
\]
for any $u \in C_0^\infty(\Om)$. The second statement of the Proposition \ref{criticalRellich} yields $2\beta + \theta \geq 2$ which contradicts to the choice of $\theta$. This finishes the proof of Theorem \ref{Maintheorem}.
\end{proof}
\section{Proof of Theorem \ref{radial}}
In this section, we give the proof of Theorem \ref{radial}. Our proof below is completely different with the one of Theorem \ref{Maintheorem}. Notice that, by scaling argument, it is enough to prove Theorem \ref{radial} on $B$. We shall prepare some ingredients for our proof. First, we have

\begin{proposition}\label{GSradial}
Let $2 \leq m \leq  \frac{n}2$ be integer and $R \geq 1$. For any radial function $u \in C_0^\infty(B)$, it holds
\begin{equation}\label{eq:GSradial}
I_{n,m}[u,B,R] \geq C_1\lt(\frac nm\rt) (n-2)^{\frac nm} R_{n,m-2,\frac nm} \int_BE_1^{\frac nm-1}\lt(\frac{|x|}R\rt) \frac{|\na v|^{\frac nm}}{|x|^{n -\frac nm}} dx,
\end{equation}
with $C_1(\frac nm) = (2^{\frac nm -1} -1)^{-1}$ and $v(x) =E_1^{\frac mn -1}(|x|/R) u(x)$.
\end{proposition}

\begin{proof}
From \eqref{eq:Rellichradial}, we have
\begin{equation}\label{eq:ab1}
\frac{I_{n,m}[u,B,R]}{R_{n,m-2,\frac nm}} \geq  \int_B \frac{|\De u|^{\frac nm}}{|x|^{n-\frac{2n}m}} dx -\lt(\frac{(n-2)(n-m)}n\rt)^{\frac nm} \int_B \frac{|u|^{\frac nm}}{|x|^n E_1^{\frac nm}(\frac{|x|}R)} dx.
\end{equation}
By polar coordinate, we have
\begin{align*}
\int_B \frac{|\na u|^{\frac nm}}{|x|^{n-\frac nm}} dx &= \om_{n-1} \int_0^1 |u'(r)|^{\frac nm} r^{\frac nm -1} dr\\
&= -\om_{n-1} \int_0^1 u''(r) |u'(r)|^{\frac nm -2} u'(r) r^{\frac nm}dr\\
&= -\om_{n-1} \int_0^1 \De u(r) |u'(r)|^{\frac nm -2} u'(r) r^{\frac nm} dr + (n-1) \om_{n-1} \int_0^1 |u'(r)|^{\frac nm} r^{\frac nm -1} dr.
\end{align*}
Then, it holds
\begin{align*}
(n-2) \int_B \frac{|\na u|^{\frac nm}}{|x|^{n-\frac nm}} dx &= \om_{n-1} \int_0^1 \De u(r) |u'(r)|^{\frac nm -2} u'(r) r^{\frac nm} dr \\
&\leq \int_B (|x| |\De u(x)| ) |\na u(x)|^{\frac nm-1} |x|^{-n + \frac nm} dx.
\end{align*}
Applying H\"older inequality, we get
\[
(n-2)^{\frac nm} \int_B \frac{|\na u|^{\frac nm}}{|x|^{n-\frac nm}} dx \leq \int_B \frac{|\De u|^{\frac nm}}{|x|^{n-\frac{2n}m}} dx.
\]
Inserting the previous inequality into \eqref{eq:ab2} yields
\begin{equation}\label{eq:ab2}
\frac{I_{n,m}[u,B,R]}{(n-2)^{\frac nm} R_{n,m-2,\frac nm}} \geq  \int_B \frac{|\na u|^{\frac nm}}{|x|^{n-\frac{n}m}} dx -\lt(\frac{n-m}n\rt)^{\frac nm} \int_B \frac{|u|^{\frac nm}}{|x|^n E_1^{\frac nm}(\frac{|x|}R)} dx.
\end{equation}
Since $u = E_1^{1-\frac mn} (|x|/R) v(x)$, we have
\[
\na u(x) = -\lt(1-\frac mn\rt) E_1^{-\frac mn}\lt(\frac{|x|}R\rt) v(x) \frac{x}{|x|^2} + E_1^{1-\frac mn}\lt(\frac{|x|}R\rt) \na v(x).
\]
Since $n\geq 2m$, we have the following inequality 
\[
|x -y|^{\frac nm} \geq |x|^{\frac nm} + C_1\lt(\frac nm\rt) |y|^{\frac nm}  - \frac nm |x|^{\frac nm -2} \la x,y\ra,
\]
for any $x, y \in \R^n$ with $C_1(\frac nm) = (2^{\frac nm-1} -1)^{-1}$ (see Lemma $3.1$ in \cite{Barbatis}). Applying this inequality, we have
\begin{align*}
|\na u(x)|^{\frac nm} &\geq  \lt(\frac{n-m}n\rt)^{\frac{n}m} \frac{|u(x)|^{\frac nm}}{|x|^{\frac nm}E_1^{\frac nm}(\frac{|x|}R)} + C_1\lt(\frac nm\rt)E_1^{\frac nm-1}\lt(\frac{|x|}R\rt) |\na v(x)|^{\frac nm}\\
&\qquad -\frac nm \lt(\frac{n-m}n\rt)^{\frac{n}m -1} \frac{|v|^{\frac nm -2} v \la\na v(x), x\ra}{|x|^{\frac nm}}.
\end{align*}
Integrating $\frac{|\na u|^{\frac nm}}{|x|^{n-\frac{n}m}} dx$ on $B$ and using integration by parts with noting that $\text{\rm div}(x/|x|^n) =0$, we get
\begin{align}\label{eq:ab3}
\int_B \frac{|\na u|^{\frac nm}}{|x|^{n-\frac{n}m}} dx & \geq \lt(1-\frac{m}n\rt)^{\frac{n}m} \int_B \frac{|u(x)|^{\frac nm}}{|x|^{n}E_1^{\frac nm}(\frac{|x|}R)} dx\notag\\
&\qquad\qquad\qquad\quad + C_1\lt(\frac nm\rt) \int_B E_1^{\frac nm-1}\lt(\frac{|x|}R\rt) \frac{|\na v(x)|^{\frac nm}}{|x|^{n-\frac nm}} dx.
\end{align}
Inserting \eqref{eq:ab3} into \eqref{eq:ab2} proves \eqref{eq:GSradial}.
\end{proof}

The next proposition is elementary.
\begin{proposition}
Let $2\leq m \leq \frac n2$ be integers and $R \geq 1$, we have
\begin{align}\label{eq:pointwise}
|u(x)| \leq  \lt(\frac{I_{n,m}[u,B,R]}{\om_{n-1}C_1\lt(\frac nm\rt) (n-2)^{\frac nm} R_{n,m-2,\frac nm}}\rt)^{\frac mn} &E_1^{1-\frac mn}\lt(\frac{|x|}R\rt)\notag\\
& \times \lt(\ln E_1\lt(\frac{|x|}R\rt)-\ln E_1\lt(\frac1R\rt)\rt)^{1-\frac mn},
\end{align}
for any radial function $u \in C_0^\infty(B)$.
\end{proposition}
\begin{proof}
Let $v(x) = E_1^{\frac mn -1}(|x|/R) u(x)$, then $v$ is radial function and $v(1) = 0$. For $0< r <1$, we have
\[
v(r) = -\int_r^1 v'(s) ds = -\int_r^1 v'(s) s E_1^{1-\frac mn}\lt(\frac sR\rt) E_1^{-1+\frac mn}\lt(\frac sR\rt) \frac{ds}s.
\]
Applying H\"older inequality, we get
\begin{align}\label{eq:pointwisev}
|v(r)| &\leq \lt(\int_r^1 |v'(s)|^{\frac nm} s^{\frac nm} E_1^{\frac nm-1}\lt(\frac sR\rt) \frac{ds}s\rt)^{\frac mn} \lt(\int_r^1 E_1^{-1}\lt(\frac sR\rt) \frac{ds}s\rt)^{\frac{n-m}n}\notag\\
&\leq \frac1{\om_{n-1}^{\frac mn}} \lt(\int_BE_1^{\frac nm-1}\lt(\frac{|x|}R\rt) \frac{|\na v|^{\frac nm}}{|x|^{n -\frac nm}} dx\rt)^{\frac mn} \lt(\ln E_1\lt(\frac sR\rt) -\ln E_1\lt(\frac1R\rt)\rt)^{\frac{n-m}n}.
\end{align}
The inequality \eqref{eq:pointwise} follows from \eqref{eq:pointwisev}, \eqref{eq:GSradial} and the definition of $v$.
\end{proof}

We are now ready to prove Theorem \ref{radial}.

\begin{proof}[Proof of Theorem \ref{radial}]
Let $u\in C_0^\infty(B)$ be a radial function such that $I_n[u,B,R] \leq 1$. Denote
\[
A_{n,m} =\om_{n-1}C_1\lt(\frac nm\rt) (n-2)^{\frac nm} R_{n,m-2,\frac nm}.
\]
By \eqref{eq:pointwise}, we have
\[
|u(x)| \leq \lt(\frac{1}{A_{n,m}}\rt)^{\frac mn} E_1^{1-\frac mn}\lt(\frac{|x|}R\rt) \lt(\ln E_1\lt(\frac{|x|}R\rt) -\ln E_1\lt(\frac1R\rt)\rt)^{1-\frac mn}.
\]
Note that $E_2(t) = 1 + \ln E_1(t)$. We claim that
\begin{equation}\label{eq:claim}
(a-\ln s) \frac{\ln(a-\ln s) -\ln a}{1 + \ln (a-\ln s)} \leq -\ln s,\quad \forall\, a\geq 1,\, s \in (0,1],
\end{equation}
which is equivalent to
\[
a \ln(a-\ln s) \leq a \ln a  - (1 + \ln a)\ln s .
\]
Using the inequality $\ln (1+ x) \leq x$ for $x> -1$, we have
\[
a \ln(a -\ln s) = a\ln a + a\ln\lt(1 -\frac{\ln s}a\rt) \leq a\ln a - \ln s \leq a \ln a  -(1 + \ln a)\ln s,
\]
since $a \geq 1$ and $s \in (0,1]$. This proves the claim \eqref{eq:claim}.

Applying the claim \eqref{eq:claim} for $a =1 + \ln R$ and $s = |x| \in (0,1)$ we get
\[
\lt(\frac{|u(x)|}{E_2^{\frac{n-m}n}\lt(\frac{|x|}R\rt)}\rt)^{\frac{n}{n-m}} \leq \lt(\frac{1}{A_{n,m}}\rt)^{\frac m{n-m}} (-\ln |x|). 
\]
Consequently, we have
\begin{align*}
\int_B e^{c \lt(\frac{|u(x)|}{E_2^{\frac{n-m}n}\lt(\frac{|x|}R\rt)} \rt)^{\frac{n}{n-m}}} dx& \leq \int_B e^{-\frac{c}{A_{n,m}^{\frac{m}{n-m}}} \ln |x|} dx  =\om_{n-1} \int_0^1 e^{-\frac{c}{A_{n,m}^{\frac{m}{n-m}}} \ln r} r^{n-1} dr < \infty,
 \end{align*}
 provided that $c < n A_{n,m}^{\frac{m}{n-m}}$. The proof is completed.
\end{proof}

\section*{Acknowledgments}
This research is funded  by the Simons Foundation Grant Targeted for Institute of Mathematics, Vietnam Academy of Science and Technology.

\end{document}